\documentclass{article}
\usepackage[margin=2cm]{geometry}

\usepackage{setspace}
\usepackage{caption}

\usepackage{import}

\usepackage{authblk}
\usepackage{amsmath, amssymb, bbold, mathtools}
\usepackage[ruled, vlined]{algorithm2e}
\usepackage{multirow}
\usepackage{graphicx}
\usepackage{tcolorbox}
\usepackage[]{hyperref}
\usepackage{bbding}
\usepackage[export]{adjustbox}
\usepackage{subcaption}

\usepackage{wasysym}
\usepackage{lmodern} 

\usepackage{natbib}
 \bibpunct[, ]{(}{)}{,}{a}{}{,}

\usepackage{pgfplots}
\usepackage{tikz}
\usetikzlibrary{calc}
\usetikzlibrary{shapes}

\usepackage{amsthm}
\theoremstyle{plain}


\newcounter{parentnumber}

\newtheorem{theorem}{Theorem}[section]
\newtheorem{lemma}[theorem]{Lemma}

\newtheorem{remark}[theorem]{Remark}

\newtheorem{proposition}[theorem]{Proposition}

\usepackage{acronym}
\acrodef{wlog}[WLOG]{without loss of generality}
\acrodef{lsc}[lsc]{lower semi-continuous}

\definecolor{red}{RGB}{163, 31, 52}
\definecolor{gray}{RGB}{194, 192, 191}
\definecolor{blue}{RGB}{59, 89, 152}
\definecolor{green}{RGB}{0, 179, 0}

\newcommand{\norm}[1]{\left\|#1\right\|}

\newcommand{\abs}[1]{\left|#1\right|}

\newcommand{\E}[2]{{\mathbb E}_{#1} \left[ #2 \right]}
\newcommand{\V}[2]{{\mathbb V}_{#1} \left[ #2 \right]}

\newcommand{\one}[1]{\mathbb{1}\left\{#1\right\}}
\renewcommand{\tfrac}[2]{{#1}/{#2}}

\renewcommand{\Pr}[0]{\mathbb P}

\newcommand{\set}[2]{\left\{ #1\ : \ #2 \right\}}
\newcommand{\tset}[2]{\{ #1\ : \ #2 \}}
\newcommand{\KL}{{\rm{KL}}}

\newcommand{\W}{{\rm{W}}}

\newcommand{\defn}[0]{:=}

\newcommand{\mc}{\mathcal}
\newcommand{\mb}{\mathbb}

\newcommand{\mr}{\mathrm}

\renewcommand{\d}{{\mathrm{d}}}

\renewcommand{\Re}{\mathrm{R}}

\newcommand{\st}{\mr{s.t.}}

\DeclareMathOperator{\TV}{TV}

\allowdisplaybreaks

\title{Robust Mean Estimation for Optimization: The Impact of Heavy Tails}
\author[1]{Bart P.G.\ Van Parys \thanks{bart.van.parys@cwi.nl}}
\author[1,2]{Bert Zwart \thanks{bert.zwart@cwi.nl}}
\affil[1]{CWI Amsterdam}
\affil[2]{TU Eindhoven}

\date{\today}

\begin{document}

\maketitle

\begin{abstract}
  We consider the problem of constructing a least conservative estimator of the expected value $\mu$ of a non-negative heavy-tailed random variable. 
  We require that the probability of overestimating the expected value $\mu$ is kept appropriately small; a natural requirement if its subsequent use in a decision process is anticipated.
  In this setting, we show it is optimal to estimate $\mu$ by solving a distributionally robust optimization (DRO) problem using the Kullback-Leibler (KL) divergence.
  We further show that the statistical properties of the proposed KL-DRO estimator compare favorably with alternative estimators based on truncation, variance regularization, or Wasserstein DRO.

\vspace{1mm}


\vspace{1mm}

\end{abstract}

\section{Introduction}

Optimization has been an indispensable tool to make informed decisions based on a mathematical model of reality.
For decisions to be practically relevant, however, it is important that model uncertainty is explicitly taken into account.
Stochastic optimization is concerned with problems of the form
\[
  \min_{x\in X} {\mb E_{\mb P}} [\ell(x, \xi)]
\]
and has classically embodied decision-making in the face of uncertainty. Stochastic optimization indeed encapsulates, through an appropriate choice of loss function $\ell$, a wide variety of problems ranging from the classical newsvendor stocking problem to supervised learning \citep{birge2011introduction,shapiro2021lectures}. 
{\color{black}
  Stochastic optimization traditionally treats the probability distribution of the exogenous uncertain variable as a given object. In practice, however, such distributions are not directly observable. Consequently, modern formulations of decision-making under uncertainty increasingly recognize that empirical data (rather than fully specified distributions) should serve as the fundamental object on which decisions are based.
}

To this end, data-driven Distributionally Robust Optimization (DRO), i.e., 
\begin{equation}
  \label{dro}
  \min_{x\in X} \max_{D({\mb Q},{\mb P}_n) \leq r} {\mb E_{\mb Q}} [\ell(x, \xi)],
\end{equation}
first considered by \cite{scarf1958min} in the context of stochastic newsvendor inventory problems has emerged over the last decade as an appealing framework, cf.\ \cite{rahimian2022frameworks} for a recent survey.
DRO seeks decisions that are optimal under the worst-case distribution in an ambiguity set $\{{\mb Q }: D({\mb Q}, {\mb P}_n) \leq r\}$ constructed as a ball of distributions sufficiently close to the empirical distribution ${\mb P}_n$ of the data according to an appropriate distance measure $D$.
Informally, DRO can hence be interpreted as a worst-case alternative to the classical sample average approximation \citep{kleywegt2002sample}.
Much of the early uptick in the popularity of DRO  may be attributed to its computational tractability.
Indeed, starting with the observation of \cite{delage2010distributionally} concerning the tractability of moment-based ambiguity sets, a wide variety of ambiguity sets have been shown to yield a tractable problem formulation (\ref{dro}).
Particularly noteworthy are formulations based on a statistical divergence \citep{bayraksan2015data} or the Wasserstein optimal transport distance \citep{pflug2007ambiguity, mohajerin2018data, gao2023distributionally, blanchet2019transport}.
Early on, the statistical properties of DRO formulations were approached indirectly by guaranteeing that their ambiguity sets serve as confidence sets for the unknown distribution \citep{delage2010distributionally, bertsimas2018data} at a desired confidence level.
That is, the probability\footnote{Note that, as customary, we slightly abuse notation in that the outer $\mb P$ in the previous equation denotes the product measure $\mb P^n$.}
\begin{equation}
  \label{eq:coverage}
  \Pr\left[D({\mb P}, {\mb P}_n) \leq r\right]
\end{equation}
is required to be sufficiently large which is typically ensured by appealing to an appropriate concentration inequality.
Indeed, it is straightforward to observe that
\begin{equation}
  \label{eq:implication}
  D({\mb P}, {\mb P}_n) \leq r \implies \forall x\in X: {\mb E_{\mb P}} [\ell(x, \xi)] \leq \textstyle  \max_{D({\mb Q},{\mb P}_n) \leq r} {\mb E_{\mb Q}} [\ell(x, \xi)]
\end{equation}
and hence the DRO formulation minimizes a safe worst-case bound on the unknown objective function.
For instance, the statistical properties of the popular Wasserstein DRO formulation, as introduced by \cite{pflug2007ambiguity}, were initially based on a classical result by \cite{dudley1969speed} who established $\E{\mb P}{W({\mb P}_n, {\mb P})}=C n^{-1/\dim(\Xi)}$ {\color{black}where $\xi\in \Xi$} for some constant $C$ assuming the distribution $\mb P$ has bounded support. Later on this result was sharpened to light-tailed distributions by \cite{mohajerin2018data} by appealing to a modern measure concentration result by \cite{fournier2015rate}.
Unfortunately, approaching the statistical properties of DRO formulations via a coverage probability guarantee has its limitations. Indeed, in case of the Wasserstein distance \cite{weed2019sharp} point out that a nonzero coverage probability (\ref{eq:coverage}) requires that the radius $r$ of the Wasserstein ambiguity set shrink no faster than the nonparametric rate $n^{-1/\dim(\Xi)}$. When working with a statistical divergence such as the Kullback-Leibler divergence, the situation becomes even more dire as for any continuous distribution $\mb P$ the coverage probability (\ref{eq:coverage}) remains zero however large the radius $r$.
{\color{black}
Consequently, the statistical properties of DRO formulations have more recently been imposed directly through
\begin{equation}
  \label{eq:coverage-pointwise}
  \Pr\left[{\mb E_{\mb P}} [\ell(x, \xi)] \leq \textstyle  \hat c_n(x) \right] \geq 1-\beta(n) \quad \forall n\geq 1
\end{equation}
with $\hat c_n(x) \defn \max_{D({\mb Q},{\mb P}_n) \leq r} {\mb E_{\mb Q}} [\ell(x, \xi)]$ the predicted worst-case cost and where $\beta(n)$ denotes a desired upper bound on the probability of a disappointment event for an arbitrary decision $x\in X$.
That is, the probability of a disappointment event in which the out-of-sample cost exceeds the predicted cost should be at most $\beta(n)$ pointwise over all decisions $x\in X$.
In fact, as we are interested in controlling the disappointment probability of a decision which minimizes the proxy cost, the stronger uniform probability
\begin{equation}
  \label{eq:coverage2}
  \Pr\left[\forall x\in X: {\mb E_{\mb P}} [\ell(x, \xi)] \leq \hat c_n(x) \right]
\end{equation}
is of primary interest.
However, using an appropriate complexity notion on the set $\mc L=\set{\xi \mapsto \ell(x, \xi)}{x\in X}$
the uniform probability \eqref{eq:coverage2} is typically controlled in terms of the pointwise probabilities \eqref{eq:coverage-pointwise} via a standard covering number argument \citep{vandervaart1996weak}.}

Under mild technical conditions, \cite{gao2023finite, blanchet2021statistical} establish that Wasserstein DRO may indeed achieve a parametric $n^{-1/2}$ rate again assuming the distribution to satisfy a light tail condition.
The statistical properties of DRO formulations based on statistical divergences have been studied in the same vein by \cite{lam2019recovering} and \cite{duchi2021statistics} who show that these formulations also enjoy parametric $n^{-1/2}$ rates by appealing to concentration results from the empirical likelihood theory by \citet{owen1988empirical} which require a finite variance condition.
Apart from several further variants of DRO, other data-driven decision procedures  exist. A well-known example, popular in the machine learning community, is variance regularization \citep{maurer2009empirical}.
{\color{black}
More generally, data-driven decisions $\hat x_n\in \arg\min \hat c_n(x)$ where $\hat c_n(x)$ is a proxy for the actual cost ${\mb E_{\mb P}} [\ell(x, \xi)]$ have been denoted as robust in this literature if they guarantee 
\begin{equation}
  \label{eq:coverage-general}
  \Pr\left[\forall x\in X: {\mb E_{\mb P}} [\ell(x, \xi)] \leq  \hat c_n(x) \right]\geq 1-\beta'(n) \quad \forall n\geq 1
\end{equation}%
where $\beta'(n)$ denotes a desired upper bound on the probability of a disappointment event for any decision $x\in X$.}
Nevertheless, the literature leaves the practitioner with an unwieldy zoo of robust decision tools
which begs the question which---if any---tool they should prefer.

In a recent line of work, the design of efficient and trustworthy decision methods was given direction by focusing on {\em least conservative} data-driven optimization procedures. 
{\color{black}
  A robust data-driven decision $\hat x_n\in \arg\min \hat c_n(x)$ comes indeed via \eqref{eq:coverage-general} with the explicit promise that its decision will cost no more than the predicted cost $\hat c_n(\hat x_n)$ with high probability. However, should we have access to another estimator $\hat c_n'$ satisfying the same guarantee \eqref{eq:coverage-general} which is \textit{less conservative} in that $\hat c_n'(x) \leq \hat c_n(x)$ for all $x\in X$ then clearly the latter should be preferred over the former as its decision $\hat x'_n\in \arg\min \hat c'_n(x)$ promises lower downstream decision costs yet enjoys the same out-of-sample disappointment guarantee.
  In \cite{vanparys2021data} it is shown based on a large deviation argument \citep{dembo2009large} that DRO formulations with a Kullback-Leibler distance function $D$ in (\ref{dro}) (henceforth abbreviated as KL-DRO) provide a \textit{least conservative} proxy $\hat c^{\KL}_n$ which among a large class of proxies $\hat c$ satisfying the out-of-sample performance guarantee (\ref{eq:coverage-general}) enjoys the least conservative property that $\limsup_{n\to\infty}\hat c_n^{\KL}(x) \leq \liminf_{n\to\infty}\hat c_n(x)$ for all $x\in X$.
While originally only established by \cite{vanparys2021data} for $\beta'(n)=\exp(-nr)$ for $r \geq 0$, an appropriately modified least conservativeness property of KL-DRO has been shown by \citet{bennouna2021learning} to extend to any $\beta'(n) \to  0$.}
This statistical efficiency of KL-DRO in terms of downstream decision performance has been shown to be resilient to the specific statistical framework considered as similar efficiency statements have been made in recent literature \citep{lam2019recovering, duchi2021statistics}.
However, as we have pointed out in this introduction, the statistical analysis of modern DRO formulations requires data with bounded support, light-tails or finite variance.
In contrast, in many practical situations, data is heavy-tailed, and sometimes even has infinite variance \citep{nair2022fundamentals}.
Heavy tails complicate the analysis of the statistical behavior of popular DRO approaches.
\cite{pflug2007ambiguity}, who seems to be the first to consider Wasserstein DRO formulations, notes already that the rate at which the Wasserstein ambiguity set based on a coverage guarantee can shrink may be arbitrarily slow if the tails of the distribution $\mb P$ are heavy enough due to a result by \cite{kersting1978geschwindigkeit}. 
The performance of data-driven decision-making algorithms in a heavy-tailed setting remains poorly understood, which this work seeks to address.

Any data-driven optimization procedure critically relies on finding a good estimator $\hat c_n(x)$ for the mean of the random loss $\ell(x, \xi)$.
When working with heavy-tailed losses, a good estimator should provide robust protection against extreme outlier losses.
A rich literature on this topic exists, and several estimators have been suggested, such as the median-of-mean estimator, $M$-estimators, truncated and/or trimmed means, and estimators based on self-normalization.
For background on such estimators, we refer to \cite{bhatt2022nearly, catoni2012challenging, Lugosi2019survey, minsker2023mom, nemirovskij1983problem, oliveira2025trimmedmean, shao1997self}.
{\color{black} Several of these estimators are known to enjoy strong statistical optimality properties: for instance, the median-of-means estimator achieves sub-Gaussian deviation bounds under only a finite variance assumption \citep{Lugosi2019survey}, meaning they guarantee $|\hat c_n(x) - \E{\mb P}{\ell(x, \xi)}| \leq C\sigma\sqrt{\log(1/\delta)/n}$ with probability $1-\delta$ at the optimal rate \citep{devroye2016subgaussian}, while \cite{bhatt2022nearly} establish minimax rate-optimal bounds under an infinite variance assumption.}
However, these optimality notions are fundamentally symmetric: over- and underestimating the expected loss are treated as equally undesirable errors.
In a decision-making context, one is in stark contrast mainly concerned with the asymmetric, one-sided error of being too optimistic \citep{mohajerin2018data, vanparys2021data, bennouna2021learning}: an overestimate of $\mu$ directly leads to a decision whose out-of-sample cost exceeds its predicted value, which has immediate consequences for the downstream optimization.

\subsection{Contributions}
\label{ssec:contributions}

{\color{black}Least conservativeness is a performance criterion tailored to a decision-making context, and distinct from the minimax deviation criterion studied in the estimation literature. The main contribution of this paper is to show that this criterion singles out KL-DRO even in a heavy-tailed setting. This will become apparent even in a setting where the loss $\ell(x, \xi)$ is bounded in one direction for a fixed decision $x$ as in the pointwise guarantee \eqref{eq:coverage-pointwise}; the stronger uniform guarantee \eqref{eq:coverage2} then follows via a standard covering number argument, as noted earlier.}
Since the decision $x$ is fixed throughout, it becomes redundant. To lighten notation and connect more naturally to statistical estimation, we henceforth write $\zeta := -\ell(x, \xi)$ for the non-negative revenue and $\mu := \E{}{\zeta}$ for its mean, and refer to $\hat e_n$ as an \emph{estimator} rather than a cost proxy. {\color{black} The least-conservativeness criterion then takes a more familiar form: among all estimators $\hat e_n$ satisfying
\begin{equation}
  \label{eq:coverage-pointwise-reversed}
  \Pr\left[\hat e_n > \mu\right] \leq \exp(-\lambda(n)) \quad \forall n \geq 1,
\end{equation}
which is the least conservative? Note that \eqref{eq:coverage-pointwise-reversed} is the equivalent to the pointwise guarantee \eqref{eq:coverage-pointwise} in the new notation with $\beta(n)=\exp(-\lambda(n))$. The exponential parameterisation is without loss of generality and natural, as such rates arise organically for KL-based methods via Cram\`er's theorem and large deviations theory \citep{dembo2009large}.} The key question is how to treat the outliers in heavy-tailed data in an optimal data-driven way.

We formalize this goal in Section \ref{sec-challenges}, and investigate how several existing methods from the literature (specifically: sample mean,  Wasserstein DRO, truncated mean, and variance regularization) behave. We focus primarily on the probability of being too optimistic (namely that our estimate exceeds $\mu$), but also the probability of being too conservative, i.e., that $\hat e_n < \mu - b$ for some $b > 0$. It turns out that
none of the listed methods provide sharp yet robust estimators of the mean $\mu$ in a purely data-driven way; in particular, truncation schemes require quantitative information about higher moments if theoretical guarantees are required; the same holds for median-of-mean estimators and $M$-estimators. 

The main results of our study are presented in Section 3, where we show in Theorem \ref{thm-klopt} that KL-DRO retains its optimality properties (previously established in a bounded support case in \cite{vanparys2021data, sutter2024pareto, bennouna2021learning}) in a heavy-tailed setting.
In addition, we show in Theorem \ref{prop-klub} that KL-DRO leads to a conservative estimate of $\mu$ with probability $1-\exp\{ - r(n) n(1+o(1))\}$ where {\color{black} $r(n)=\lambda(n)/n$} substitutes for the robustness radius $r$ in formulation (\ref{dro}). Finally, in Proposition \ref{prop-kllb} we show that the probability of being conservative by a term $b>0$ decreases exponentially fast in the number of data points $n$ --- surprisingly at the same rate as for the non-robust sample mean (which is recovered when $r=0$ in (\ref{dro})).
These insights show that KL-DRO remains the premier robust method of choice even with heavy-tailed data and generalizes previous work of \cite{bennouna2021learning} who assumed the support of $\zeta$ to be discrete. \cite{bennouna2021learning} also shows that in this discrete setting, KL-DRO and variance regularization are asymptotically equivalent in settings where the radius in (\ref{dro}) decays to zero, i.e., $r=r(n)\rightarrow 0$.
We show here that this observation no longer holds in the heavy-tailed case.
Specifically, as we illustrate in Section \ref{sec-challenges} variance regularization may not optimally handle outliers, as it {\em overcompensates} for them, leading to unnecessarily conservative estimates. 

We prove our results using a variety of different mathematical techniques which may be of independent interest. The proofs of the results in Section \ref{sec-challenges} are all rather short and based on arguments from the analysis of heavy-tailed random variables and large-deviations theory; they are presented in Appendix \ref{sec-proofs} to keep the paper self-contained. 
The proofs of our main results in Section \ref{sec-kldro} are likely of more intrinsic interest. In particular, our upper bound for the probability of being too optimistic is based on an improvement of a concentration bound of \cite{cappe2013kullback}, and a similar result of \cite{agrawal2021regret}; our key lemma in this respect is Lemma \ref{lemma:upper-bound-log-log-nonasymptotic}. 
The proof of Theorem \ref{thm-klopt} follows by contradiction, and leverages dual properties of KL-DRO. In particular, Proposition \ref{lemma:variance-bound} provides a detailed asymptotic analysis of the variance of the dual solution using techniques from the theory of regularly varying functions. 

If losses are unbounded and in particular heavy-tailed in both directions, the associated analysis will be more challenging, and the nature of the results potentially different. In fact, at a minimum it would require imposing a structural assumption on the distribution ${\mb P}$ such as the existence of a finite moment of order $a>1$ as considered in related work on heavy-tailed bandits \cite{agrawal2021regret}. Our focus on losses being bounded in one direction does not require any such structural assumption and allows for an elegant mathematical treatment.

\subsection{Outline}

In Section \ref{sec-challenges} we give a careful introduction of our setting, 
and give an overview of four methods from the literature. Our main results on KL-DRO can be found in Section \ref{sec-kldro}. Proofs of our results are presented in Sections \ref{sec-proofub}, \ref{sec-proofoptimality}, and Appendix \ref{sec-proofs}.

\subsection{Notation} 

We say that a random variable $\zeta$ is regularly varying of index $\rho>1$ if  
\begin{equation}
\label{eq-assumptiontail}
    {\mb P}[\zeta > u] = L(u)u^{-\rho}, 
\end{equation}
with $\rho>1$ and $L(u)$ slowly varying, i.e.\ $L(au)/L(u)\rightarrow 1$ as $u\rightarrow\infty$ for $a>0$.  A function $L$ is slowly varying if and only if
  \begin{equation}
    \label{eq:karamata-representation}
    L(u) = \exp\left(\delta(u) + \int_1^u \frac{\epsilon(t)}{t} \d t\right) \quad \forall u\geq 1
  \end{equation}
  for a measurable function $\delta(u)$  with $\lim_{u\to\infty} \delta(u) < \infty$ and $\epsilon(t)$ a measurable function with $\lim_{t\to\infty} \epsilon(t)=0$ {\cite[Theorem 1.3.1]{bingham1989regular}}.
We say that a random variable $\zeta$ is regularly varying if it is regularly varying of some index $\rho>1$.
We write $f(x)/g(x)\rightarrow 0$ as $f(x)= o(g(x))$. 
We denote the probability simplex on $\Re_+$ as $\mc P$.
Given $n$ independent samples $\zeta_1,...,\zeta_n$ from a common random variable $\zeta$ with distribution $\mb P\in \mc P$, we denote their empirical measure by  $\mb P_n$. The expected value and variance of random variables under ${\mb P}$ are written by $\E{\mb P}{\cdot}$ and $\V{\mb P}{\cdot}$. In particular, the sample mean and sample variance of  $f(\zeta)$ (with $f$ deterministic) are written as $\E{\mb P_n}{f(\zeta)}$ and $\V{\mb P_n}{f(\zeta)}$.
In particular, we have
\begin{equation}
\label{eq-mean}
    \hat \mu_n  := \E{\mb P_n}{\zeta} = \textstyle \frac 1n \sum_{i=1}^n \zeta_i
\mbox{ and }
  \hat \sigma_n^2  := \V{\mb P_n}{\zeta} = \textstyle \frac 1n \sum_{i=1}^n \left(\zeta_i- \frac 1n \sum_{j=1}^n \zeta_j\right)^2.  
\end{equation}

\section{Challenges with existing approaches}
\label{sec-challenges}

\subsection{Disappointment and bias}
As mentioned in the introduction, our goal is to examine robust estimators of the mean reward $\mu = {\mb E}[\zeta]$.
Given $n$ i.i.d.\ samples $\zeta_1,...,\zeta_n$ and their associated empirical measure $\mb P_n$ our aim is to construct an estimator $\hat e_n =\hat e_n( \mb P_n)$ which safely estimates $\mu$ in an optimal way despite the fact that, as we work with heavy-tailed data, outliers are expected.
In this section, we will not yet define what ``optimal'' means (for this, refer to Theorem \ref{thm-klopt} below), but rather focus instead on reviewing several existing methods. 

We examine these methods 
through two specific performance metrics. 
The first property relates to the probability of the estimator being {\em too optimistic}. In particular, we say that
an estimator $\hat e_n$ is {\em feasible at a disappointment rate $\lambda(n)$} if it holds that
\begin{equation}
  \label{eq:exponential-disappointment-rate}
  \limsup_{n\to\infty}  \frac {1}{\lambda(n)} \log \Pr \left[\hat e_n > \mu \right] \leq -1
\end{equation}
{\color{black}which is the asymptotic equivalent of (\ref{eq:coverage-pointwise-reversed}).}
That is, the rate $\lambda(n)$ imposes a speed of decay of the probability that our estimator overestimates the unknown expected reward on a logarithmic scale and hence disappoints. This statistical requirement has been imposed in previous works. In particular, in \cite{vanparys2021data, liu2023smoothed} it was required that $\lambda (n) = rn$ for fixed $r>0$, while in \cite{bennouna2021learning} sublinear and superlinear scalings were considered. It was shown in \cite{bennouna2021learning} that superlinear scalings of $\lambda (n)$ require $\hat e_n\to \max \zeta =\infty$ and is hence not of interest here. Consequently, we will concern ourselves with the case that $\lim_{n\to\infty} \lambda (n)/n < \infty$. Finally, \cite{duchi2021statistics} and \cite{lam2019recovering} considered the case $\lambda(n)=O(1)$. Although we do not rule out this case entirely, our main result in Section \ref{sec-kldro} does require $\lambda(n)\gg \log\log(n)$.

Constructing estimators that are feasible as defined in the previous paragraph, may yield trivially conservatively biased estimators (i.e., $\hat e_n\equiv 0$).
Our goal in the next section is to construct a ``least conservative'' estimator.
In the present section, however, we review several existing estimators satisfying (\ref{eq:exponential-disappointment-rate}) and examine the probability that an estimator is {\em conservatively} biased by a term $b$.
That is, we examine the {\em conservatism probability} $\Pr [\hat e_n < \mu-b]$ for $b>0$; the faster this probability converges to $0$, the less conservative and more preferred the estimator.

In the next four subsections, we will investigate the performance of several well-known estimators according to these two metrics. 

\subsection{Sample mean estimators}
\label{sec:sample-mean-pred}

We will consider here perhaps the simplest estimator $\hat e_{n,\Delta} = \hat\mu_n-\Delta$ by reducing the sample mean $\hat \mu_n = \frac 1n \sum_{i=1}^n \zeta_i$ with a bias term  {\color{black}$\Delta>0$}.
Using existing results from the literature, we can assess the feasibility and conservativeness of $\hat e_{n,\Delta}$. For instance, \cite{wang2008sample} investigates this estimator using a standard large deviation argument requiring a finite moment generating function.
However, a standard result (see for example
 \citet[Theorem 3.11]{nair2022fundamentals}) is that for regularly varying rewards $\zeta$ of index $\rho>1$ we have
 \begin{equation}
   \label{eq:catastrophe}
    \Pr[ \hat e_{n,\Delta}> \mu] = n {\mb P}[\zeta > \Delta n] (1+o(1)) = \exp \left( - (1+o(1)) (\rho-1) \log n\right).
\end{equation}
{\color{black} Hence, if the disappointment rate $\lambda(n)$ satisfies $\liminf_{n\to\infty} \lambda (n)/ \log n> \rho-1$ then no nontrivial sample mean estimators are feasible in (\ref{eq:exponential-disappointment-rate}). In this regime, a deflated sample mean $\hat e_{n,\Delta(n)}$ is only feasible when $\Delta(n)\to\infty$ sufficiently fast and hence by the law of large numbers the estimates are trivial as $\hat e_{n,\Delta(n)}\to -\infty$ almost surely. 
The reason is that the sample mean is very sensitive to even a single outlier and such outliers occur far more frequently than in a light tailed setting.}
In addition, this class of estimators is by construction biased by a factor $\Delta$ and hence conservative. {\color{black} We illustrate this finding on a Pareto example in Section~\ref{sec:numerics} (Figure~\ref{fig:mean-vs-kl}).}

{\color{black}The sample mean $\hat e_{n,0}=\hat \mu_n$ itself is not conservative, but rather too optimistic, as it does not guarantee (\ref{eq:exponential-disappointment-rate}) for any $\lambda(n)\uparrow \infty$.}
We can apply Cram\`ers theorem \citep[Theorem 2.2.3]{dembo2009large} to show that its conservatism probability satisfies
\begin{equation}
\label{sample-mean-left-tail}
-\frac1n  \log  \Pr[ \hat e_{n,0} < \mu -b] \rightarrow I(b) := \sup_{s>0} \, (b-\mu)s-\log \E{\mb P}{ \exp(-s\zeta ) } .
\end{equation}
Since $b>0$ and $\zeta\geq 0$, the supremum characterizing the exponential rate $I(b)$ can be restricted to non-negative $s$ for which $\E{\mb P}{ \exp(-s\zeta )}<\infty$, and it follows that $I(b)>0$.
{\color{black} We include the sample mean as a reference in Figure~\ref{fig:conservatism} (Section~\ref{sec:numerics}) to illustrate this Cram\`er rate.}
Perhaps surprisingly, we will see later it is possible to construct an estimator that achieves feasibility at any desired sublinear rate $\lambda(n)$ while not being more conservative than the sample mean.

The previous result shows that, in a heavy-tailed setting, the simple estimation strategy of deflating the empirical mean is inadequate. Perhaps a fair criticism of this observation would be that such estimators are far too simplistic to be taken seriously. However, any estimator which is larger than the estimator $\hat e_{n,\Delta}$ can trivially also not be feasible either. 

\subsection{Optimal transport estimators}
\label{ssec:wass-dro-estim}

For instance,  consider an optimal transport estimator which is defined as the solution to the optimization problem
\begin{align}
  \label{def:W-estimator}
\hat e^{\W}_{n,\gamma} \defn & \left\{
  \begin{array}{rl}
    \min & \int u \,\d \mb Q(u)\\
    \st & \mb Q\in \mc P, ~\W({\mb P}_n, \mb Q)\leq \gamma.
  \end{array}\right.
\end{align}
where $W$ denotes an optimal transport metric, i.e.,
\(
  \W(\hat{\mb P}, \mb Q) = \inf \tset{\int d(u, v) \, \d {\mb T}(u, v)}{\mb T\in \Gamma(\hat{\mb P}, \mb Q)}
\)
where $d:\Re\times\Re\to\Re_+$ denotes a convex lower semicontinuous transport cost with $d(u,v)=0$ for $u=v$ and where the transport polytope of all joint distributions $\mb T$ with marginals $\hat{\mb P}$ and $\mb Q$ is denoted as $\Gamma(\hat{\mb P}, \mb Q)$.
This class of optimal transport metrics has been a popular choice in distributionally robust optimization, due to their appealing computational properties and solid statistical guarantees in light-tailed settings \citep{blanchet2019transport, gao2023distributionally, mohajerin2018data, kuhn2019wasserstein}.
{\color{black}
To examine its statistical properties in a heavy-tailed setting, observe first that for the classical choice $d(u,v)=\norm{u-v}$ associated with the Wasserstein metric \citep{villani2008optimal}, Kantorovich-Rubinstein duality \cite[p.\ 4]{kuhn2019wasserstein} gives $\hat e^{\W}_{n,\gamma} = \max(0,\hat e_{n,\gamma})$, and hence $\Pr[\hat e^{\W}_{n,\gamma}>\mu] = \Pr[\hat e_{n,\gamma}>\mu]$, which can be observed numerically in Figure~\ref{fig:mean-vs-kl}.
More generally, the following proposition shows that optimal transport estimators with continuous convex transport costs, such as those considered in \citet{gao2023distributionally, mohajerin2018data, kuhn2019wasserstein}, exhibit the same catastrophic sensitivity to heavy-tailed data as the sample mean.

\begin{proposition}
  \label{prop-wass-lb}
  Let $\zeta$ be regularly varying of index $\rho>1$ and let $d$ be a continuous convex transport cost with $d(u,v)=0$ for $u=v$ and $d(\mu+\Delta,\mu)>0$ for some $\Delta>0$.
  Then for any fixed $\gamma\geq 0$,
  \[
    \Pr\left[\hat e^{\W}_{n,\gamma}>\mu\right] \geq \exp\left(-(1+o(1))(\rho-1)\log n\right).
  \]
\end{proposition}

The proof is given in Appendix~\ref{sec-proofs}.
Hence, if the disappointment rate $\lambda(n)$ satisfies $\liminf_{n\to\infty} \lambda (n)/ \log n> \rho-1$ then no nontrivial optimal transport estimators with convex cost is feasible in (\ref{eq:exponential-disappointment-rate}). In this regime, such estimators are only feasible when $\gamma(n)\to\infty$ sufficiently fast, but then the estimates are trivial: since $\zeta\geq 0$ we have $\hat e^{\W}_{n,\gamma}\geq 0$ always, while if $\E{}{d(\zeta,0)}<\infty$ then $\W(\mb P_n,\delta_0)=\frac{1}{n}\sum_{i=1}^n d(\zeta_i,0)\to \E{\mb P}{d(\zeta,0)}$ almost surely by the law of large numbers, so $\delta_0$ eventually becomes feasible as $\gamma(n)\to\infty$, giving $\hat e^{\W}_{n,\gamma(n)}\to 0$.
}

{\color{black}
  The total variation distance which is associated with the nonconvex transport cost $d(u, v)=\one{u\neq v}$ \citep{shapiro2017distributionally, bennouna2022holistic} is not subject to the negative result in Proposition \ref{prop-wass-lb}. Denote its corresponding estimates as
  \begin{align*}
    \hat e^{\TV}_{n, \gamma} \defn & \left\{
                                \begin{array}{rl}
                                  \min & \int u \,\d \mb Q(u)\\
                                  \st & \mb Q\in \mc P, ~\TV(\mb P_n, \mb Q)\leq \gamma.
                                \end{array}\right.
  \end{align*}
  We will show in Section \ref{sec-kldro} that such total variation estimates are robust to heavy-tailed data. That is, using a judiciously scaled robustness radius $\gamma(n)$, they are both nontrivial and satisfy the statistical guarantee required in Equation (\ref{eq:exponential-disappointment-rate}) as long as $\lambda(n)/\log \log n\rightarrow\infty$. We discuss these total variation estimates in the context of the numerical example in Figure \ref{fig:conservatism} (Section \ref{sec:numerics}).
}

\subsection{Truncated sample mean estimators}
\label{ssec:trunc-sample-mean}

Assume we are given constants $a\in (1,2],A<\infty$ such that ${\mb E} [ \zeta^a]\leq A$. 
In this setting, it is possible to construct a feasible estimator which is unbiased  and as conservative as the sample mean, if $\lambda(n) = o(n)$. This can be achieved by appropriately truncating large values of $\zeta_i$. We summarize our results in the following proposition and refer to Section \ref{sec-proofs} for a proof. 
For similar results, we refer to  \cite{bubeck2013bandit}. For convenience, we state the cases $a=2$ (finite variance) and $a\in (1,2)$ (infinite variance) separately. 

\begin{proposition} 
  \label{prop-trim-mean}
  Let $r(n) = \lambda(n)/n$. Then,
  \begin{itemize}
  \item[(i)] For $a=2$, define 
    $C = \max \{1/4, Ae^2/2\}$. Then,
    \begin{equation}
      \Pr \left[\frac 1n \textstyle\sum_{i=1}^n (\zeta_i \wedge (1/\sqrt{r(n)})) -  2\sqrt{C r(n)}   > \mu \right]\leq  e^{-\lambda(n)} \quad \forall n\geq 1.
    \end{equation}
  \item[(ii)] 
    For $a\in (1,2)$, 
    define 
    $C= \max \{\frac{1}{(a-1)a^a}, A e^a\frac{2}{2+a}\}$ and
    $c_a= (a-1)^{-(a-1)/a} a C^{1/a}$.
    Then,
    \begin{equation}
      \Pr \left[\frac 1n \textstyle\sum_{i=1}^n (\zeta_i \wedge r(n)^{-1/a}) - c_a  r(n)^{(a-1)/a} > \mu \right]
      \leq  e^{-\lambda(n)} \quad \forall n\geq 1.
    \end{equation}
  \end{itemize}
\end{proposition}

We now define the truncated mean estimator 
\begin{equation}
    \hat e_{n}^{a,A} = \frac 1n \sum_{i=1}^n (\zeta_i \wedge r(n)^{-1/a}) - c_a r(n)^{(a-1)/a} 
\end{equation}
and conclude using Proposition \ref{prop-trim-mean} that $\hat e_{n}^{a,A}$ indeed satisfies the out-of-sample disappointment requirement (\ref{eq:exponential-disappointment-rate}).
Moreover, it can be verified (by applying 
Lemma \ref{lemma-truncatedlefttail} with $u(n)= r(n)^{-1/a}$ and $v(n)= c_a r(n)^{(a-1)/a}$)
that 
the  probability $\Pr[\hat e_{n}^{a,A} < \mu - b]$ of being conservative by a term $b$ decays to 0 at exponential rate $I(b)$, which is as fast as the sample mean
(\ref{sample-mean-left-tail}). This makes the truncated mean estimator an appealing estimator. However, the truncated mean estimator is not
entirely data-driven as knowledge of the constants $a$ and $A$ is required. A similar remark can be made concerning other estimators of the sample mean which we do not discuss in detail here, such as  median-of-means  \cite{nemirovskij1983problem}
and the $M$-estimator suggested by \cite{catoni2012challenging}.
{\color{black} We include the truncated mean in the numerical comparison of Section~\ref{sec:numerics} (Figure~\ref{fig:conservatism}).}

\subsection{Variance regularization}
\label{ssec:vari-regul}

Imagine for the sake of argument that the random returns have finite variance $\V{\mb P}{\zeta}=\sigma^2<\infty$ and consider here once more the sample mean estimator $\hat e_{n, \Delta}$. The central limit theorem establishes under a finite variance assumption that
\(
  \sqrt{n} \left(\hat \mu_n - \mu\right) \to N(0, \sigma^2)
\)
in distribution. One could hope via the informal line of argument 
\[
   \Pr \left[\hat e_n > \mu \right] = \Pr \left[\sqrt{n}(\hat \mu_n- \mu) > \sqrt{n} \Delta \right]\approx \frac{1}{\sqrt{2\pi \sigma^2}}\int^\infty_{\sqrt{n}\Delta} \exp(-u^2/(2\sigma^2)) \d u\leq \exp(\tfrac{-n\Delta^2}{(2\sigma^2)})
\]
that by setting the tolerance as $\Delta = \sqrt{2 \lambda(n)/n} \sigma$ the sample mean estimator $\hat e_{n, \Delta}$ should be feasible in (\ref{eq:exponential-disappointment-rate}). Leaving aside the fact that in general the variance is not defined, there are still two issues with the presented approach.
The first issue is that, as discussed in Section \ref{sec:sample-mean-pred}, the informal line of argument {\color{black} deceptively} fails to hold for $\mb P$ regularly varying of index $\rho$ and disappointment rates satisfying $\liminf_{n\to\infty} \lambda (n)/ \log n> \rho-1$.
Informally, we observed that the empirical mean is too sensitive to outliers and when exposed to heavy-tailed data tends to \textit{overestimate} the unknown $\mu$ too often to be feasible in (\ref{eq:exponential-disappointment-rate}). A second issue, mirroring our remarks from the previous section, is that this estimator is not entirely data driven as it  requires knowledge of the variance.
A straightforward fix to the second issue is naively substituting the empirical variance for the unknown variance $\sigma^2$.
Let indeed
\begin{equation}
  \label{def:var_reg}
\hat e^{vr}_n = \hat \mu_n -   \sqrt{2 r(n)} \hat\sigma_n,     
\end{equation}
with $r(n) = \lambda(n)/n$ and $\hat \sigma^2_n$ the sample standard variance defined in Equation (\ref{eq-mean}). The statistical properties of this variance regularized estimator have been studied by \cite{maurer2009empirical} and more recently by \citet{duchi2019variance,lam2019recovering,bennouna2021learning}.

Surprisingly, the proposed straightforward substitution also fixes the first sensitivity issue of the empirical mean. Recall that the failure of the empirical mean is fundamentally due to the fact that a single catastrophically large observation can cause it to overestimate the unknown $\mu$.
Let us see how the variance regularization behaves if one of the data points is as catastrophic as the outliers considered in Equation \eqref{eq:catastrophe}. Consider an event where $\zeta_i= d n$, for some $d>0$ and some $i\leq n$. 
In this case, the sample mean $\hat \mu_n$ is increased by $d$, but this is offset by the compensating term $\sqrt{2 r(n)} \hat\sigma_n$. Some elementary computations show 
that this term reduces $\hat e^{vr}_n$ by a term of the order $\sqrt{2\lambda(n)} d$. 
Thus, to reduce $\hat e^{vr}_n$ by $b>0$, a single large value
$b \sqrt{n/(2r(n))}$
of the summands suffices; this is made rigorous in the proof of Proposition \ref{prop-variance-lb} below.
We see that the way outliers are compensated is somewhat excessive. Hence, catastrophically large values in the sample do not lead to too optimistic values of our estimator, but rather {\em too conservative} values. 

{\color{black}
  To make these insights rigorous, we partially rely on limit theory for self-normalized random sums of \cite{shao1997self} who establishes that
  \begin{equation}
    \label{eq:shao-norm-sums}
    \Pr\left[\tfrac{\left(\mu_n-\mu\right)}{\textstyle\sqrt{\frac 1n \sum_{i=1}^n (\zeta_i-\mu)^2}} > \sqrt{2\lambda(n)/n} \right] = e^{-\lambda (n) (1+o(1))}
  \end{equation}
  for any sequence $\lambda(n)\uparrow \infty$ and $\lambda(n)=o(n)$ as long as $x\mapsto \E{}{\zeta^2 \one{\abs{\zeta}\leq x}}$ is slowly varying.
  However, in the light of Equation \eqref{def:var_reg} we will be interested in normalizing using the empirical variance $\hat \sigma_n$ instead.
  The proof of the following result is found in Appendix \ref{sec-proofs} and is based on combining Equation \eqref{eq:shao-norm-sums} with the bias-variance decomposition $\hat \sigma^2_n +  (\mu_n-\mu)^2 = \sum_{i=1}^n (\zeta_i-\mu)^2/n$.
}

\begin{proposition}
\label{prop-variance-ub}
{Assume $\E{}{\zeta^2}<\infty$, $\lambda(n)\uparrow \infty$} and $r(n) = \lambda(n)/n\rightarrow 0$.
  We have
  \[
    \Pr[\hat e^{vr}_n > \mu ] = e^{-\lambda (n) (1+o(1))}.
  \]   
\end{proposition}

We now provide theoretical support for
the insight that the probability of being conservative
by a term $b$ can be rather large. {\color{black} The proof of the following proposition is also deferred to Appendix \ref{sec-proofs}.}

\begin{proposition}
\label{prop-variance-lb}
Assume ${\mb E}[\zeta^2]<\infty$ and $r(n)\rightarrow 0$. We have
\begin{equation}
    \label{VR-LB}
        \Pr[\hat e^{vr}_{n} < \mu- b] \geq (1+o(1)) n {\mb P} [\zeta > b \sqrt{n/(2r(n))}].
    \end{equation}
\end{proposition}
In particular, whenever $\mb P$ is regularly varying of index $\rho$ and $r(n) = n^{-\beta}, \beta \in (0,1)$, we see that $$\Pr[\hat e^{vr}_{n} < \mu- b] \geq  e^{- (1+o(1)) (\frac\rho 2 (1+\beta)-1) \log n }.$$
This slow polynomial decay contrasts with the exponential decay exhibited in case of the (truncated) sample mean, and our KL-DRO estimator in Section \ref{sec-kldro}; {\color{black}see also Figure~\ref{fig:conservatism} for a numerical illustration.}
Hence, although variance regularization offers more robust protection against heavy-tailed losses than empirical mean estimation, it seems to do so inefficiently in terms of the resulting estimation bias.

\section{The KL-DRO estimator}

\label{sec-kldro}
The previous section reviewed different methods to estimate the mean $\mu$, focusing on the (disappointment) probability of overestimating the reward and
the probability of being inefficient by a term $b>0$. 
None of the methods we discussed were able to reach both objectives in a purely data-driven way. In this section, we will introduce the KL-DRO estimator, which is entirely data-driven, and show that this estimator performs well on both metrics.
More importantly, we provide theoretical support for the claim that KL-DRO is, in a sense, optimal. 

The KL-DRO estimator is defined as the solution to the optimization problem
\begin{align}
  \label{eq:kl-predictor}
\hat e^{\KL}_{r}(\hat{\mb P}) \defn & \left\{
  \begin{array}{rl}
    \min & \int u \,\d \mb Q(u)\\
    \st & \hat {\mb P}\ll \mb Q, ~\int \log(\tfrac{\d \hat {\mb P}}{\d \mb Q}(u)) \d \hat{\mb P}(u) \leq r
  \end{array}\right.
\end{align}
where the constraint $\hat{\mb P} \ll \mb Q$ requires the distribution $\hat{\mb P}$ to be absolutely continuous with respect to $\mb Q$. We now show that this estimator satisfies the desired disappointment guarantee when its radius is scaled appropriately as
\(
r(n) \defn \lambda(n)/n
\)
and write $\hat e^{\KL}_{n}=\hat e^{\KL}_{r(n)}({\mb P}_n)$. 

\begin{theorem}
\label{prop-klub}
If $\lambda(n)/\log \log n\rightarrow\infty$, then
\begin{equation}
  \label{eq:kl-u-guarantee}
  {\mb P} [\hat e^{\KL}_{n} > \mu] \leq e^{-\lambda (n) (1+o(1))}.
    \end{equation}
\end{theorem}

{\color{black} We illustrate the exponential decay of the KL-DRO disappointment probability in Figure~\ref{fig:mean-vs-kl} (Section~\ref{sec:numerics}).
We prove this result in Section \ref{sec-proofub}, by deriving the non-asymptotic bound
\begin{equation}
  \label{eq-ub-nice}
  {\mb P} [\hat e^{\KL}_{n} > \mu] \leq (e \lambda (n) \log n + e^2)  e^{-\lambda (n)}
\end{equation}
for any $\mb P \in \mc P$ and $n\geq 2$ so that $\lambda(n)> 1$.} In fact, the previous inequality implies a slightly stronger uniform guarantee
\(
  \sup_{\mb P\in \mc P} {\mb P} [\hat e_n > \E{\mb P}{\zeta}] = e^{-\lambda (n) (1+o(1))}.
\)
This non-asymptotic bound which holds for any $\lambda(n)>1$, formally presented in Lemma \ref{lemma:upper-bound-log-log-nonasymptotic}, is in line with similar results in the literature (\cite{agrawal2021regret}, \cite{cappe2013kullback}), where the term in front of the exponent $e^{-\lambda (n)}$ involves a term which is linear in $n$. Our bound is therefore sharper if $\lambda (n) \log n$ increases at a slower-than-linear rate. If we would have applied the results from 
\cite{agrawal2021regret}, \cite{cappe2013kullback}
directly, we should have made the more stringent assumption $\lambda(n)/\log n\rightarrow \infty$, ruling out disappointment probabilities that decrease at polynomial rate. We prove Theorem \ref{prop-klub} by modifying the argument in \cite{cappe2013kullback}, as is explained in more detail in Section \ref{sec-proofub}.

Our main result is that KL-DRO is in a sense optimal:  

\begin{theorem}
\label{thm-klopt}
Consider an estimator $\hat e_n=\hat e_n({\mb P}_n)$ which satisfies
\begin{equation}
  \label{eq:uniform-exponential-guarantees}
  \sup_{\mb P\in \mc P} {\mb P} [\hat e_n > \E{\mb P}{\zeta}] = e^{-\lambda (n) (1+o(1))}.
\end{equation}
That is, the same guarantee \eqref{eq-ub-nice} as the Kullback-Leibler estimator for some $\lambda(n)\rightarrow\infty$. 
 Then, for any regularly varying $\zeta$ we have
  \[
    \liminf_{n\to\infty} {\mb P}[\hat e_n \leq \hat e^\KL_{r'(n)}(\mb P_n)]  = 1,
  \]
  for any radius $r'(n)$ so that $\lim_{n\to\infty}\tfrac{r'(n) n}{\lambda(n)} \in (0, 1)$.
\end{theorem}

Informally, the previous result says that any competing estimator 
must be at least as conservative as a slightly aggressively scaled KL-DRO formulation whenever $\zeta$ is sufficiently well behaved. In particular, it holds for any regularly varying $\zeta$, and as we discuss in Section \ref{sec-proofoptimality}, extends straightforwardly to bounded $\zeta$ too.

We prove this result in Section \ref{sec-proofoptimality}, using a dual representation of KL-DRO at the distribution $\mb P$. In this dual representation, the Radon-Nikodym derivative between the distribution $\mb P$ and the maximizer characterizing the KL-DRO estimator $e^{\KL}_{r}(\mb P)$ in (\ref{eq:kl-predictor}) plays a prominent role.  A key step (Proposition \ref{lemma:variance-bound}) is to show that the variance of its logarithm is of the order $r(n)$. For this, we have to find appropriate decompositions of various terms that make up this variance, and estimate these terms using techniques from the theory of regularly varying functions. The proof of Theorem \ref{thm-klopt} itself is by deriving a contradiction: we show that an estimator performing better than KL-DRO with positive probability cannot be feasible, which is done by using an exponential change-of-measure argument.

{\color{black}
  Theorem \ref{thm-klopt} is a \emph{relative} result and indicates that KL-DRO is the least conservative estimator enjoying the guarantee \eqref{eq:uniform-exponential-guarantees}.
  The following Proposition \ref{prop-kllb} will provide a complementary \emph{absolute} statement: KL-DRO does not pathologically underestimate $\mu$.
   KL-DRO reduces to the sample mean for the special case $r=0$. Hence, KL-DRO is smaller than the sample mean whenever $r>0$.
  We furthermore remark that also the TV-DRO estimates $\hat e^{\TV}_{n}=\hat e^{\TV}_{n, \gamma(n)}$ with $\gamma(n)=\sqrt{r(n)/2}$ also satisfy Equation (\ref{eq:uniform-exponential-guarantees}) as they are more conservative than the KL-DRO estimates, i.e., we have from Pinsker's inequality \citep{csiszar2011information} that $\hat e^{\TV}_{n}\leq \hat e^{\KL}_{n}$.
 The probability that either DRO estimates are conservative by a term $b>0$, surprisingly, has a similar behavior to that of the sample mean:
 \begin{proposition}
   \label{prop-kllb}
   If $b>0$ and $r(n)=\lambda(n)/n\rightarrow 0$, then 
   \begin{equation}
     {\mb P} [\hat e^{\KL}_{n} < \mu-b] \leq {\mb P} [\hat e^{\TV}_{n} < \mu-b] \leq e^{- n I(b)(1+o(1))}.
   \end{equation}
 \end{proposition}
}
The proof of this proposition can be found in Appendix \ref{sec-proofs}.

If the random rewards $\zeta_i$'s have bounded support, it is shown in \cite{bennouna2021learning} that KL-DRO and variance regularization are equivalent as $n\rightarrow\infty$. Comparing Propositions \ref{prop-variance-lb} and \ref{prop-kllb} shows this is no longer the case under heavy-tailed data, as KL-DRO behaves better. In particular, KL-DRO resists outliers in a more effective (less conservative) way than variance regularization.

The above two results show that KL-DRO has good statistical properties ``in both directions'': it is safe with high probability, but also not too conservative with high probability, as long as $r(n)\rightarrow 0$. {\color{black} Both properties are illustrated numerically in Figures~\ref{fig:mean-vs-kl} and~\ref{fig:conservatism} (Section~\ref{sec:numerics}).}

{\color{black}
\begin{remark}[General Divergence Estimators]
\label{remark:diverg-estim}

The KL divergence belongs to the broader family of $f$-divergence estimators, which replace the KL ball in Equation~\eqref{eq:kl-predictor} with $\set{\mb Q}{D_f(\mb P_n,\mb Q)=f((\d\mb Q/\d \mb P_n)(u))\d \mb P_n(u) \leq r}$ for a convex function $f:\Re_+\to\Re_+\cup\{+\infty\}$ with $f(1)=0$.
When $f(0)<\infty$, as for the total variation distance ($f(t)=|t-1|/2$, already discussed above) or the Pearson divergence ($f(t)=(t-1)^2/2$, studied as a convex proxy to variance regularization by \citet{duchi2019variance}), the estimator can zero out extreme observations \citep{love2015phi}, providing natural protection against heavy-tailed outliers.
Indeed, this outlier-suppression property is the mechanism underlying the feasibility of TV-DRO established above, as it plays a central role in the proof of Proposition~\ref{prop-kllb}.

The KL divergence, with $f(t)=-\log(t)+(t-1)$ and $f(0)=+\infty$, cannot suppress outliers in this way.
Its optimality in our framework stems from a different and deeper property: the KL divergence is the rate function in Sanov's theorem, the large deviations principle for empirical distributions \citep{dembo2009large}.
It is precisely this connection that drives Theorems \ref{prop-klub} and \ref{thm-klopt}. While other $f$-divergences with $f(0)<\infty$, such as TV, give feasible estimators, Theorem \ref{thm-klopt} shows that any such estimator is at least as conservative as KL-DRO, so KL-DRO is optimal in our setting.

\end{remark}
}

{\color{black}
\section{A Numerical Illustration}
\label{sec:numerics}

We illustrate our theoretical findings on a simple concrete example in which $\zeta$ is independent and Pareto distributed with scale parameter $x_m=1$ and shape parameter $\alpha=3$, i.e., $\Pr[\zeta>x] = (1/x)^\alpha$ for all $x\geq 1$.
This distribution is regularly varying of index $\rho=\alpha=3$ with mean $\mu = \E{}{\zeta} = 3/2$, variance $\sigma^2 = 3/4$, and second moment $\E{}{\zeta^2} = 3 < \infty$.

\paragraph{Disappointment probability.}
Figure~\ref{fig:mean-vs-kl} compares the disappointment probability $\Pr[\hat e_n > \mu]$ for the deflated sample mean estimator $\hat e_{n,\Delta}$ with $\Delta=0.096$, the Wasserstein DRO estimator $\hat e^\W_{n,\Delta}$, and the KL-DRO estimator $\hat e^{\KL}_{r}(\mb P_n)$ with $r=0.011$, as a function of sample size $n$, using Monte Carlo simulation.
As discussed in Section \ref{sec-challenges}, the disappointment probability of the deflated sample mean and Wasserstein DRO estimators coincide and decays subexponentially (sublinearly on a logarithmic scale) due to their catastrophic sensitivity to outliers, while the KL-DRO estimator achieves an exponential decay at rate $\exp(-rn)$ despite having the same asymptotic bias as the deflated sample mean and Wasserstein estimates.

\begin{figure}[h]
    \centering
    \begin{subfigure}{0.45\textwidth}
\begin{tikzpicture}
\begin{axis}[
  scale=0.9,
  xlabel={$n$},
  title={Disappointment $\mathbb P[\hat e_n > \mu]$},
  grid=both,
  ymode=log,
  ymin=1e-5,
  xmin=0,
  xmax=1000,
  legend to name=meankllegend,
  legend columns=4,
  legend style={/tikz/every even column/.append style={column sep=0.5cm}}
  ]

  \addplot+[mark=none] table[
  col sep=comma,
  x=n,
  y=disappointment
  ] {figures/kl-vs-mean/kl-vs-mean-kl-mean.csv};
  \addlegendentry{KL-DRO}

  \addplot+[mark=none] table[
  col sep=comma,
  x=n,
  y=disappointment
  ] {figures/kl-vs-mean/kl-vs-mean-wasserstein-mean.csv};
  \addlegendentry{W-DRO}

  \addplot+[mark=none] table[
  col sep=comma,
  x=n,
  y=disappointment
  ] {figures/kl-vs-mean/kl-vs-mean-inflated-mean.csv};
  \addlegendentry{Mean ($\hat\mu_n - \Delta$)}

  \addplot+[mark=none, black, densely dotted] table[
  col sep=comma,
  x=n,
  y=rate
  ] {figures/kl-vs-mean/rate.csv};
  \addlegendentry{$\exp(-r n)$}

\end{axis}
\end{tikzpicture}
\caption{Disappointment}
    \label{fig:mean-vs-kl-d}
  \end{subfigure}
  \hfill
    \begin{subfigure}{0.45\textwidth}
\begin{tikzpicture}
\begin{axis}[
  scale=0.9,
  xlabel={$n$},
  title={Bias $\mu-\E{}{\hat e_n}$},
  grid=both,
  ymin=0,
  xmin=0,
  xmax=1000
  ]

  \addplot+[mark=none] table[
  col sep=comma,
  x=n,
  y=b
  ] {figures/kl-vs-mean/kl-vs-mean-kl-mean.csv};

  \addplot+[mark=none] table[
  col sep=comma,
  x=n,
  y=b
  ] {figures/kl-vs-mean/kl-vs-mean-wasserstein-mean.csv};

  \addplot+[mark=none] table[
  col sep=comma,
  x=n,
  y=b
  ] {figures/kl-vs-mean/kl-vs-mean-inflated-mean.csv};

\end{axis}
\end{tikzpicture}
  \caption{Bias}
    \label{fig:mean-vs-kl-b}
  \end{subfigure}

  \medskip
  \pgfplotslegendfromname{meankllegend}
  \caption{Three mean estimators on Pareto data, all calibrated to share the same bias $\Delta$. The deflated mean ($\hat\mu_n-\Delta$) and W-DRO estimators have constant bias and their disappointment probability decays sublinearly (on a logarithmic scale). The KL-DRO estimator shares the same asymptotic bias but achieves an exponential decay in disappointment probability at rate $r$.}
  \label{fig:mean-vs-kl}
\end{figure}

\paragraph{Conservatism probability.}

We design four estimators to be feasible in (\ref{eq:exponential-disappointment-rate}) for $\lambda(n)=\sqrt{n}/10$:
\begin{itemize}
\item The truncated mean estimator $\hat e_n^{a,A}$ feasible due to Proposition \ref{prop-trim-mean} using $(a,A)=(2,3)$.
\item The variance regularized estimator $\hat e_n^{vr}$ feasible due to Proposition \ref{prop-variance-ub}.
\item The KL-DRO estimator $\hat e_n^{\KL}$ feasible due to Theorem \ref{prop-klub}.
\item The TV-DRO estimator $\hat e^{\TV}_{n}$ feasible due to $\hat e^{\TV}_{n}\leq \hat e^{\KL}_{n}$; see discussion preceding Proposition \ref{prop-kllb}.
\end{itemize}
Figure~\ref{fig:conservatism} compares the probability $\Pr[\hat e_n < \mu - b]$ of underestimating the mean by more than $b=0.3$ for several estimators.
For reference, we also provide the conservatism probability $\Pr[\hat \mu_n < \mu - b]$ of the empirical mean which as we noted in Section \ref{sec:sample-mean-pred} is not feasible in (\ref{eq:exponential-disappointment-rate}). By Cram\`er's theorem the empirical mean is conservative by an amount $b=0.3$ with probability which decays exponentially at rate $I(b)=0.19$.
As established in Proposition~\ref{prop-kllb}, both the KL-DRO and TV-DRO estimators match this exponential rate, while variance regularization decays only sublinearly as established in Proposition~\ref{prop-variance-lb}.
The truncated mean estimator $\hat e_n^{a,A}$ also achieves the Cram\`er rate $I(b)$ asymptotically (by Lemma~\ref{lemma-truncatedlefttail}), though this is only apparent at substantially larger sample sizes than those shown due to the slow vanishing of its bias correction term.
As is clear from the results in Figure \ref{fig:conservatism}, the KL-DRO estimator also appears empirically less conservative than the TV-DRO estimator.

\begin{figure}[ht]
  \centering
\begin{tikzpicture}
\begin{axis}[
  scale=1,
  xlabel={$n$},
  ylabel={$\mathbb P[\hat e_n < \mu - b]$},
  grid=both,
  ymode=log,
  legend style={
    at={(1.02,1)},
    anchor=north west,
    cells={anchor=west}
  },
  ymin=3e-5,
  xmin=0,
  xmax=4000
  ]

  \addplot+[mark=none] table[
  col sep=comma,
  x=n,
  y=disappointment
  ] {figures/kl-vs-var/kl-vs-var-tr-mean.csv};
  \addlegendentry{Truncated mean}

  \addplot+[mark=none] table[
  col sep=comma,
  x=n,
  y=disappointment
  ] {figures/kl-vs-var/kl-vs-var-var-reg-mean.csv};
  \addlegendentry{Variance regularization}

  \addplot+[mark=none] table[
  col sep=comma,
  x=n,
  y=disappointment
  ] {figures/kl-vs-var/kl-vs-var-kl-mean.csv};
  \addlegendentry{KL-DRO}

  \addplot+[mark=none] table[
  col sep=comma,
  x=n,
  y=disappointment
  ] {figures/kl-vs-var/kl-vs-var-tv-mean.csv};
  \addlegendentry{TV-DRO}
  
  \addplot+[mark=none, densely dotted] table[
  col sep=comma,
  x=n,
  y=disappointment
  ] {figures/kl-vs-var/kl-vs-var-mean.csv};
  \addlegendentry{Mean ($\hat \mu_n$)}

  \addplot+[mark=none, densely dotted] table[
  col sep=comma,
  x=n,
  y=rate
  ] {figures/kl-vs-var/rate.csv};
  \addlegendentry{Rate ($\exp(-n I(b))$)}
  
\end{axis}
\end{tikzpicture}
  \caption{Conservatism probability $\Pr[\hat e_n < \mu - b]$ with $b=0.3$ for five estimators on Pareto($\alpha=3$) data with $\lambda(n)=\sqrt{n}/10$. The KL-DRO and TV-DRO estimators match the Cram\`er rate $\exp(-nI(b))$ of the sample mean (linearly on a logarithmic scale). The truncated mean also achieves this rate asymptotically but converges slowly. Variance regularization decays only sublinearly (subexponentially).}
  \label{fig:conservatism}
\end{figure}
}

\section{Proof of Theorem \ref{prop-klub}}
\label{sec-proofub}

{\color{black}
We first give a proof sketch to aid the reader; the complete proof follows below.
The proof proceeds in three steps.
First, we reduce the problem to bounding $\Pr[{\rm KL}_{\inf}(\mb P_n,\mu)\geq r(n)]$, where
\[
  {\rm{KL}}_{\inf}(\mb P_n, \mu) = \inf \set{\KL(\mb P_n, \mb Q)}{\mb Q\in \mc P, ~\E{\mb Q}{\zeta}\leq \mu}
\]
is the infimal KL divergence from $\mb P_n$ to the set of distributions with mean at most $\mu$. By definition of $\hat e^{\KL}_{r(n)}$, we have
\begin{equation*}
{\mb P}\left[\hat e^{\KL}_{r(n)} > \mu \right] \leq
   {\mb P}\left[ {\rm{KL}}_{\inf}(\mb P_n, \mu) \geq r(n) \right].
\end{equation*}
Second, we exploit the dual representation of \citet{honda2012finite}, which expresses ${\rm KL}_{\inf}(\mb P_n,\mu)$ as the maximum of a concave function $\phi(\alpha)$ over $\alpha\in[0,1]$.
Third, we bound $\Pr[\max_{\alpha\in[0,1]}\phi(\alpha)\geq r(n)]$ by partitioning $[0,1]$ using a \emph{geometrically} spaced grid $\{(1-\gamma(n))^l\}_{l=0}^{m(n)}$.
For each interior interval $[(1-\gamma(n))^l,(1-\gamma(n))^{l-1}]$, concavity of $\phi$ and the secant inequality reduce the maximum over the interval to the value at its left endpoint $(1-\gamma(n))^l$, up to the factor $(1-\gamma(n))$; a Chernoff bound then contributes $e^{-n(1-\gamma(n))r(n)}$.
For the tail interval $[0,(1-\gamma(n))^{m(n)}]$, we appeal to \citet[Lemma 7]{cappe2013kullback} followed by another Chernoff bound.
Choosing $\gamma(n)=1/\lambda(n)$ and $m(n)=\lambda(n)\log n$ balances the two contributions and yields the bound in Lemma~\ref{lemma:upper-bound-log-log-nonasymptotic}.
The geometric spacing is the key improvement over \citet{cappe2013kullback}, who use a linear grid: with a geometric grid the number of intervals needed to reach $(1-\gamma(n))^{m(n)}\leq 1/2$ is only $m(n)=O(\lambda(n)\log n)$, giving a prefactor $O(\lambda(n)\log n)$ in the bound rather than $O(n)$.
}

\begin{lemma}
    \label{lemma:upper-bound-log-log-nonasymptotic}
   For every $m(n)$ such that
 $(1-1/\lambda(n))^{m(n)}\leq 1/2$ and $\lambda (n) > 1$, we have, with $r(n) = \lambda (n)/n$,
     \begin{equation}
     \label{eq-klub-hard}
        \Pr\left[{\rm{KL}}_{\inf}(\mb P_n, \mu)\geq r(n)\right] \leq e^{-\lambda (n)} \left[m(n) e+ \exp(2n e^{-m(n)/\lambda (n)}) \right].
    \end{equation}
    In particular, if additionally $n\geq 2$ we can take $m(n) = \lambda (n) \log n$ to get
   \begin{equation}
   \label{eq-klub-simple}
        \Pr\left[{\rm{KL}}_{\inf}(\mb P_n, \mu)\geq r(n)\right] \leq e^{-\lambda (n)} \left[ e  \lambda (n) \log n + e^2 \right].
    \end{equation}
    
      \end{lemma}

Using this lemma, we prove Theorem \ref{prop-klub}: 

\begin{proof}[Proof of Theorem  \ref{prop-klub}]
    Set $m(n) = \lambda(n) \log n $.
    As $\lambda(n)\rightarrow\infty$, for $n$ large enough, this choice of $\lambda(n)$ satisfies the condition needed for the inequality (\ref{eq-klub-simple}). 
Taking logarithms in   (\ref{eq-klub-simple}) and  using the inequality $a+b \leq 2 \max \{a, b\}$, we get 
\begin{equation}
    \frac 1{\lambda(n)} \log \Pr\left[{\rm{KL}}_{\inf}(\mb P_n, \mu)\geq r(n)\right] \leq    \frac 1{\lambda(n)} \left[-\lambda (n) + \log 2 + \max\{ 1 + \log \lambda (n) + \log \log n,  2\} \right],
\end{equation}
which converges to $-1$ since  $\lambda(n)/\log \log n\rightarrow\infty$.    
\end{proof}
        
  \begin{proof}[Proof of Lemma \ref{lemma:upper-bound-log-log-nonasymptotic}.]
    If $\zeta$ is degenerate at 0, the result is trivial,  so we exclude this case from now on. 
    Consider the concave function $\phi(\alpha, \zeta)\defn \log\left(1-\alpha \tfrac{(\mu-\zeta)}{\mu}\right)$.
    A straightforward dualization argument \citep[Equation 2]{honda2012finite} shows that
    \begin{equation}
    \label{eq-klinfrepresentation}
        {\rm{KL}}_{\inf}(\mb P_n, \mu) = \max_{0\leq \alpha\leq 1} \left\{\phi(\alpha) \defn  \textstyle \frac 1n \sum_{k=1}^n \phi(\alpha, \zeta_k) \right\} \geq 0.
    \end{equation}
          The remainder of this proof is inspired by the proof of \citet[Lemma 6]{cappe2013kullback} who employ the Chernoff bound based on partitioning the dual feasible set $[0,1]$ in (\ref{eq-klinfrepresentation}) using a linearly spaced grid. We instead propose to use a {\color{black} geometrically} spaced grid.

    Let $\gamma(n)\in (0, 1)$ and $m(n)\in \mb N$ so that $(1-\gamma(n))^{m(n)}\leq 1/2$.
    Using a simple union bound we can establish
    \begin{align*}
      \Pr\left[{\rm{KL}}_{\inf}(\mb P_n, \mu) \geq r(n)\right] \leq & \sum_{l=1}^{m(n)} \Pr\left[\max_{(1-\gamma(n))^l\leq \alpha \leq (1-\gamma(n))^{l-1}} \phi(\alpha) \geq r(n)\right] +  \Pr\left[\max_{0\leq \alpha\leq (1-\gamma(n))^{m(n)}} \phi(\alpha) \geq r(n)\right].
    \end{align*}
    We now study each of the terms separately in the upper bound. Observe that from the secant inequality applied to the concave function $\phi$ we have for any $(1-\gamma(n))^l\leq \alpha\leq (1-\gamma(n))^{l-1}$ that 
    \begin{align*}
      \phi((1-\gamma(n))^l)\geq & \frac{(1-\gamma(n))^l}{\alpha} \phi(\alpha) + \left(1-\frac{(1-\gamma(n))^l}{\alpha}\right)\phi(0) = \frac{(1-\gamma(n))^l}{\alpha} \phi(\alpha).
    \end{align*}
    If the event $\max_{(1-\gamma(n))^l\leq \alpha\leq (1-\gamma(n))^{l-1}} \phi(\alpha)\geq 0$ occurs, then clearly we have for $(1-\gamma(n))^l\leq \alpha\leq (1-\gamma(n))^{l-1}$ the inequality
    \begin{align*}
      \phi((1-\gamma(n))^l) \geq & \frac{(1-\gamma(n))^l}{\alpha} \phi(\alpha)\geq (1-\gamma(n)) \phi(\alpha).
    \end{align*}
    Hence, the event $\max_{(1-\gamma(n))^l\leq \alpha\leq (1-\gamma(n))^{l-1}} \phi(\alpha)\geq 0$ implies $\phi((1-\gamma(n))^l) \geq (1-\gamma(n)) \max_{(1-\gamma(n))^l\leq \alpha\leq (1-\gamma(n))^{l-1}}\phi(\alpha) $.
    
    Hence, we have for any $l\in [1,\dots, m(n)]$ that
    \begin{align*}
      & \Pr\left[\max_{(1-\gamma(n))^l\leq \alpha\leq (1-\gamma(n))^{l-1}} \phi(\alpha) \geq r(n)\right]\\
      =& \Pr\left[(1-\gamma(n))\max_{(1-\gamma(n))^l\leq \alpha\leq (1-\gamma(n))^{l-1}} \phi(\alpha) \geq (1-\gamma(n)) r(n),~\max_{(1-\gamma(n))^l\leq \alpha\leq (1-\gamma(n))^{l-1}} \phi(\alpha)\geq 0 \right]\\
      \leq & \Pr\left[\phi((1-\gamma(n))^l) \geq (1-\gamma(n)) r(n),~\max_{(1-\gamma(n))^l\leq \alpha\leq (1-\gamma(n))^{l-1}} \phi(\alpha)\geq 0 \right]\\
      \leq & \Pr\left[\phi((1-\gamma(n))^l) \geq (1-\gamma(n)) r(n)\right].
    \end{align*}
    By the Chernoff bound we have for any $l\in [1,\dots, m(n)]$ that
    \begin{align*}
      \Pr\left[\phi((1-\gamma(n))^l) \geq (1-\gamma(n)) r(n)\right] = & \Pr\left[\frac 1n \sum_{k=1}^n \log\left(1-(1-\gamma(n))^l \frac{\mu-\zeta_k}{\mu}\right) \geq (1-\gamma(n)) r(n)\right]\\
      \leq & \exp(-n(1-\gamma(n)) r(n))\E{\mb P}{\prod_{k=1}^n \left(1-(1-\gamma(n))^l \frac{\mu-\zeta_k}{\mu}\right)}\\
      = & \exp(-n(1-\gamma(n)) r(n))\prod_{k=1}^n \E{\mb P}{\left(1-(1-\gamma(n))^l \frac{\mu-\zeta_k}{\mu}\right)}\\
      = & \exp(-n(1-\gamma(n)) r(n)),
    \end{align*}
    where the first equality uses $\gamma(n)<1$ and the ultimate equality uses the fact that $\E{\mb P}{\zeta_k}=\mu$ for all $k$. We have shown that
    \[
      \sum_{l=1}^{m(n)} \Pr\left[\max_{(1-\gamma(n))^l\leq \alpha\leq (1-\gamma(n))^{l-1}} \phi(\alpha) \geq r(n)\right]\leq m(n) \exp(-n(1-\gamma(n)) r(n)).
    \]

    We now derive an upper bound on the last term $\Pr\left[\max_{0\leq \alpha\leq (1-\gamma(n))^{m(n)}} \phi(\alpha) \geq r(n)\right]$ based directly on the argument made in \citet[Lemma 6]{cappe2013kullback}.
    First, we remark that
    \(
    \tfrac{(\mu -\zeta)}{\mu}\leq 1
    \)
    for all $\zeta\geq 0$. Using $(1-\gamma(n))^{m(n)}\leq 1/2$ from \citet[Lemma 7]{cappe2013kullback} (with $\alpha = \lambda \leq \lambda'=(1-\gamma(n))^{m(n)}\leq 1/2$) we have 
    \begin{align*}
      \max_{0\leq \alpha\leq (1-\gamma(n))^{m(n)}} \phi(\alpha) = & \max_{0\leq \alpha\leq (1-\gamma(n))^{m(n)}} \frac 1n \sum_{k=1}^n \log\left(1-\alpha \frac{\mu-\zeta_k}{\mu}\right) \\
      \leq  & \frac 1n \sum_{k=1}^n \log\left(1-(1-\gamma(n))^{m(n)} \frac{\mu-\zeta_k}{\mu}\right) + 2 (1-\gamma(n))^{m(n)}\\
      \leq & \phi((1-\gamma(n))^{m(n)})+2(1-\gamma(n))^{m(n)}.
    \end{align*}
    Again exploiting the Chernoff bound we obtain
    \begin{align*}
      & \Pr\left[\phi((1-\gamma(n))^{m(n)}) \geq r(n)-2 (1-\gamma(n))^{m(n)} \right]\\
      = & \Pr\left[\frac 1n \sum_{k=1}^n \log\left(1-(1-\gamma(n))^{m(n)} \frac{\mu-\zeta_k}{\mu}\right) \geq r(n)-2 (1-\gamma(n))^{m(n)} \right]\\
      \leq & \exp(-n (r(n)- 2 (1-\gamma(n))^{m(n)}))\E{\mb P}{\prod_{k=1}^n \left(1-(1-\gamma(n))^{m(n)} \frac{\mu-\zeta_k}{\mu}\right)}\\
      = & \exp(-n (r(n)- 2 (1-\gamma(n))^{m(n)}))\prod_{k=1}^n \E{\mb P}{\left(1-(1-\gamma(n))^{m(n)} \frac{\mu-\zeta_k}{\mu}\right)}\\
      = & \exp(-n (r(n)- 2 (1-\gamma(n))^{m(n)}))
    \end{align*}
    from which we can conclude
    \begin{equation}
        \Pr\left[{\rm{KL}}_{\inf}(\mb P_n, \mu)\geq r(n)\right] \leq m(n) \exp(-n(1-\gamma(n)) r(n))+\exp(-n (r(n)- 2 (1-\gamma(n))^{m(n)})).
    \end{equation}
   The inequality (\ref{eq-klub-hard}) now follows by taking $\gamma (n) = 1/\lambda (n) \in (0, 1)$ as $\lambda(n)> 1$ and using the inequality $(1-1/\lambda)^\lambda \leq 1/e$ for $\lambda>1$. The inequality (\ref{eq-klub-simple}) 
   follows immediately from (\ref{eq-klub-hard}).
  \end{proof}

\section{Proof of Theorem \ref{thm-klopt}}

\label{sec-proofoptimality}

{\color{black}
We first give a proof sketch to aid the reader; the complete proof follows below.
The proof is by contradiction and proceeds in two main steps, covered in Sections~\ref{sec-proofoptimality:step1} and \ref{ssec:disapp-lower-bound} respectively.

\textit{Step 1: Variance bound for the dual log-likelihood ratio (Propositions~\ref{lemma:finite:dual} and~\ref{lemma:variance-bound}).}
The KL-DRO estimator in \eqref{eq:kl-predictor} admits a dual representation (Proposition~\ref{lemma:finite:dual}) in which its optimal distribution $\mb Q^r$ has log-likelihood ratio $\eta^r(u) = \log\bigl((\alpha^r+u)/\nu^r\bigr)$ with respect to $\mb P$ — a log-linear tilt parameterized by a dual variable $\alpha^r > 0$ with $\nu^r= \exp(\int \log(\alpha^r+u) \, \d \mb P(u)  -r)$.
The key technical result (Proposition~\ref{lemma:variance-bound}) is that the variance $\V{\mb P}{\eta^r(\zeta)}/r$ remains bounded as $r\to 0$.
To prove this, we reparameterize by $\tilde\alpha = 1/\alpha^r \to 0$ and study the ratio $\V{\mb P}{\log(1+\tilde\alpha\zeta)}/\tilde r(\tilde\alpha)$, where
\[
  \tilde r(\tilde\alpha) \defn \int \log(1 + \tilde\alpha u)\,\d\mb P(u) + \log\int \frac{1}{1+\tilde\alpha u}\,\d\mb P(u)
\]
satisfies $r = \tilde r(\tilde\alpha)$ by the KKT conditions of Problem~\eqref{eq:finite:dual}.
Lemma~\ref{lemma-ubvariancelog} gives an upper bound on the variance by splitting at $\tilde\alpha\zeta = 1$, and Lemma~\ref{lemma:r_lower_bound} gives a matching lower bound on $\tilde r(\tilde\alpha)$.
Both bounds involve terms of order $\tilde\alpha^\rho L(1/\tilde\alpha)$ arising from the heavy tail, controlled via the Karamata representation of $L$ and the uniform convergence theorem for slowly varying functions.
The ratio remains bounded because numerator and denominator grow at the same rate.

\textit{Step 2: Contradiction via exponential change of measure.}
Suppose an estimator $\hat e$ satisfies Equation~\eqref{eq:uniform-exponential-guarantees} uniformly over $\mb P\in\mc P$ but outperforms KL-DRO with nonvanishing probability $p>0$ (Equation~\eqref{eq:suboptimal}).
Set $r(n) = (1-\delta/2)\lambda(n)/n$ for a small $\delta>0$ and let $\mb Q^{r(n)}$ be the optimal distributions in Problem~\eqref{eq:kl-predictor} evaluated at $\mb P$.
These are ``frustrating distributions'': they lie close to $\mb P$ in KL divergence ($\KL(\mb P, \mb Q^{r(n)}) = r(n)$) and have mean $\mu^{r(n)} = \hat e^\KL_{r(n)}(\mb P) = \int u\,\d\mb Q^{r(n)}(u)$, so the event of interest is $\hat e_n > \mu^{r(n)}$.
An exponential change of measure with tilting $\eta^{r(n)} = \log(\d\mb P/\d\mb Q^{r(n)})$ converts back from $\mb Q^{r(n)}$ to $\mb P$, yielding a lower bound on $\mb Q^{r(n)}[\hat e_n > \mu^{r(n)}]$ with two factors: an exponential cost $\exp(-nr(n))$ from the Radon-Nikodym weight, and a probability under $\mb P$ that the estimator exceeds $\mu^{r(n)}$ and the empirical mean of $\eta^{r(n)}$ concentrates around its expectation.
The concentration event holds with probability tending to $1$ by Chebyshev's inequality together with Step~1, which ensures $\V{\mb P}{\eta^{r(n)}(\zeta)} = O(r(n)) = o(1)$.
A lower bound on ${\rm KL}_{\inf}(\mb P_n, \mu^{r(n)})$ via the Donsker--Varadhan representation then shows that when the concentration event holds, the event in Equation~\eqref{eq:suboptimal} implies $\hat e_n > \mu^{r(n)}$, so $\mb P[\hat e_n > \mu^{r(n)}]\to p > 0$.
Combining these, we obtain $\mb Q^{r(n)}[\hat e_n > \mu^{r(n)}]\geq \exp(-(1-\delta/2)\lambda(n)(1+o(1)))$, which decays strictly more slowly than $\exp(-\lambda(n))$.
Since $\mb Q^{r(n)}\in\mc P$, this contradicts the assumption that $\hat e$ satisfies Equation~\eqref{eq:uniform-exponential-guarantees} uniformly.
}

\subsection{Properties of the dual solution}
\label{sec-proofoptimality:step1}

We state the following well known dual characterization of the Kullback-Leibler formulation for the sake of completeness. We give a short proof for completeness and leave the reader with the remark that very similar duality results are established in \citet{bayraksan2015data, honda2012finite, ahmadi2012entropic, agrawal2020optimal, agrawal2021regret}.

\begin{proposition}[]
  \label{lemma:finite:dual}
  For any $r>0$, we have the dual characterization
  \begin{equation}
    \label{eq:finite:dual}
   \hat e^{\KL}_r(\mb P)
    = \left\{
      \begin{array}{rl}
        \max &  \exp(-r)\exp\left(\int \log(\alpha+u)\, \d \mb P(u)\right)-\alpha\\
        \st & \alpha\geq 0.
      \end{array}\right.
  \end{equation}
  Furthermore,  we have that
  \[
    \eta^r(u)\defn \log\left(\frac{\d \mb P}{\d \mb Q^r}(u)\right) = \log\left(\frac{\alpha^r+u}{\nu^r}\right) 
  \]
  with $\nu^r= \exp(\int \log(\alpha^r+u) \, \d \mb P(u)  -r)>0$ where $\mb Q^r$ and $\alpha^r$ are optimal primal and dual solutions in Equations ~(\ref{eq:kl-predictor}) and \eqref{eq:finite:dual}, respectively.
\end{proposition}

\begin{proof}

  \citet[Lemma 5]{vanparys2021data} proves\footnote{They establish this result assuming a compact event set but their proof applies verbatim to an event set $\Re_+$ which is bounded from below.} via a standard dualization argument that
  \[
    \hat e^\KL_{r}(\mb P) = \max_{\alpha\geq 0}~ \exp(-r)\exp\left(\textstyle\int \log(\alpha+u)\, \d \mb P(u)\right)-\alpha,
  \]
  where the maximum is achieved at some $\alpha^r\geq 0$. 
  
  Introduce the positive measure \( \mb Q^r_c \) through its Radon-Nikodym derivative
  \[
    \frac{\d \mb Q^r_c}{\d \mb P}(u) = \frac{\nu^r}{\alpha^r+u}.
  \]
  A necessary optimality condition for the optimal dual solution $\alpha^r$ in Problem \eqref{eq:finite:dual} is
  \begin{align}
    \label{eq:kkt-dual}
    -1 +\exp(-r)\exp\left(\int \log(\alpha^r+u)\, \d{\mb P}(u)\right)\int \frac{1}{\alpha^r+u} \d {\mb P}(u) + q^r_s = 0
  \end{align}
  for some $q^r_s\geq 0$ satisfying the complementarity condition $\alpha^r q^r_s = 0$. We can rewrite the previous optimality condition as
  \begin{equation*}
    \eqref{eq:kkt-dual} \iff \int \frac{\nu^r}{\alpha^r+u} \d {\mb P}(u) + q^r_s \leq 1 \iff \int \d \mb Q^r_c(u)+q^r_s = 1.
  \end{equation*}
  We claim now that an optimal primal solution in (\ref{eq:kl-predictor}) can be found as
  $\mb Q^r = \mb Q^r_c+q^r_s \delta_{0}$. By construction $\mb Q^r$ is in $\mc P$ as indeed it is positive and satisfies $\int \d \mb Q^r(u) = \int \d \mb Q^r_c(u) + q^r_s = 1$. Furthermore, we have that 
  \[
    \KL(\mb P, \mb Q^r) = \int \log\left(\frac{\d \mb P}{\d (\mb Q^r_c+q^r_s\delta_{0})}(u)\right) \d \mb P(u) \leq \int \log\left(\frac{\d \mb P}{\d \mb Q^r_c}(u)\right) \d \mb P(u) = \int \log\left(\alpha^r+u\right) \d \mb P(u) - \log(\nu^r) = r.
  \]
  Finally, observe that
  \begin{align*}
    \int u \d \mb Q^r(u) = \int u \d \mb Q_c^r(u) = -\alpha^r +\int (u+\alpha^r) \frac{\nu^r}{\alpha^r+u} \d \mb P(u) = -\alpha^r+\nu^r=-\alpha^r+\exp\left(\int \log(\alpha^r+u) \, \d \mb P(u)  -r\right),
  \end{align*}
  guaranteeing that $\mb Q^r$ is optimal in (\ref{eq:kl-predictor}).
  If $\alpha^r>0$, then by complementarity slackness we have $q^r_s=0$ and 
  \[
    \frac{\d \mb P}{\d \mb Q^r}(u) =  \frac{\d \mb P}{\d \mb Q_c^r}(u) = \frac{\alpha^r+u}{\nu^r}.
  \]
  Otherwise, if $\alpha^r=0$ then simply
  \[
    \frac{\d \mb P}{\d \mb Q^r}(u) = \frac{\d \mb P}{\d (\mb Q_c^r+q_s^r \delta_{0})}(u) =
    \left\{
    \begin{array}{rl}
      \frac{\alpha^r+u}{\nu^r} & {\rm{if}}~u> 0\\
      0 & {\rm{if}}~u=0
    \end{array}\right\}
    = \frac{\alpha^r+u}{\nu^r}
  \]
  and the proof is complete.
\end{proof}

\begin{proposition}
  \label{lemma:variance-bound}
  Let $\zeta$ be regularly varying. Then,
  \[
    \lim_{r\to 0}\frac{\V{\mb P}{\eta^{r}(\zeta)}}{r} < \infty.
  \]
\end{proposition}
\begin{proof}
  Recall that as $L$ is slowly varying it admits the Karamata representation
  \[
    L(u) = \exp\left(\delta(u) + \int_1^u \frac{\epsilon(t)}{t} \d t\right)
  \]
  with $\delta(u)$ a bounded measurable function with $\lim_{u\to\infty} \delta(u) < \infty$ and $\epsilon(t)$ a bounded measurable function with $\lim_{t\to\infty} \epsilon(t)=0$ \cite[Theorem 1.3.1]{bingham1989regular}. We may hence consider $\tilde \alpha_0 \in (0, 1]$ sufficiently small  so that
  \(
  \max_{t\geq 1/\tilde \alpha_0}\epsilon(t)\leq \rho/2
  \)
  and
  \(
  \max_{t\geq 1/\tilde \alpha_0} \delta(t)-\delta(1/\tilde \alpha_0)\leq \log(2).
  \)
  First, we define for the sake of convenience for any $\tilde\alpha \in (0,  \tilde \alpha_0]$ the function
  \begin{align*}
    \tilde r(\tilde \alpha) \defn & \int \log(1 + \tilde \alpha u)  \d \mb P(u) + \log \int \frac{1}{1 + \tilde \alpha u} \d \mb P(u)
  \end{align*}
  which is nonnegative due to Jensen's inequality as the logarithm is a concave function. In fact, following Lemma \ref{lemma:r_lower_bound} stated below, we can lower bound $\tilde r$ for any $\tilde \alpha \to 0$ as
  \begin{equation}
    \tilde r(\tilde \alpha) \geq \frac{1}{32} \V{\mb P}{\min(\tilde \alpha \zeta,1)} +\tilde \alpha^{\rho}L(1/\tilde \alpha) \int^K_{1} \frac{z}{(1+z)^2 z^{\rho}} \d z + o\left(\tilde \alpha^{\rho}L(1/\tilde \alpha) \right)
  \end{equation}
  for any $K>1$. Second, we have the following upper bound, as stated by Lemma \ref{lemma-ubvariancelog} below: 
  \begin{equation}
    \V{\mb P}{\log(1 + \tilde \alpha \zeta)}\leq \V{\mb P}{\min (\tilde \alpha \zeta, 1))}  + \tilde \alpha^{\rho} L(1/\tilde \alpha)\int^\infty_{1} \frac{4\log(z+1)}{(z+1)z^{\rho/2}}\d z.
  \end{equation}

Consequently,
\begin{align*}
  \frac{\V{\mb P}{\log(1 + \tilde \alpha \zeta)}}{\tilde r(\tilde \alpha)}\leq &  \frac{\V{\mb P}{\min(\tilde \alpha \zeta, 1)}  + \tilde \alpha^{\rho} L(1/\tilde \alpha)\int^\infty_{1} \frac{4\log(z+1)}{(z+1)z^{\rho/2}} \d z}{\frac{1}{32} \V{\mb P}{\min(\tilde \alpha \zeta, 1)} +\tilde \alpha^{\rho}L(1/\tilde \alpha) \int^K_{1} \frac{z}{(1+z)^2 z^{\rho}} \d z + o\left(\tilde \alpha^{\rho}L(1/\tilde \alpha)\right)}\\
  \leq & 32+\frac{\int^\infty_{1} \frac{4\log(z+1)}{(z+1)z^{\rho/2}}  \d z}{\int^K_{1} \frac{z}{(1+z)^2 z^{\rho}}\d z +o(1) }
\end{align*}
for all $K> 1$. Hence, we have
\[
  \lim_{\tilde \alpha\downarrow 0}  \frac{\V{\mb P}{\log(1 + \tilde \alpha \zeta)}}{\tilde r(\tilde \alpha)}\leq 32+\frac{\int^\infty_{1} \frac{4\log(z+1)}{(z+1)z^{\rho/2}}  \d z}{\int^\infty_{1} \frac{z}{(1+z)^2 z^{\rho}}\d z} < \infty.
\]

We are finally ready to show the desired result.
A necessary optimality condition for the dual solution $\alpha^r$ in Problem \eqref{eq:finite:dual} is
  \begin{align*}
    0 \leq & 1 -\exp(-r)\exp\left(\int \log(\alpha^r+u)\, \d{\mb P}(u)\right)\int \frac{1}{\alpha^r+u} \d {\mb P}(u)\\
    \iff r\geq &  \int \log(\alpha^r+u)\, \d {\mb P}(u) +\log \int \frac{1}{\alpha^r+u} \d {\mb P}(u).
  \end{align*}
  
  As the function $\alpha \mapsto \int \log(\alpha+u)\, \d {\mb P}(u) +\log \int \tfrac{1}{(\alpha+u)} \d {\mb P}(u)$ is nonincreasing in $\alpha$ (from the convexity of the dual objective function in Problem \eqref{eq:finite:dual}) it follows that for $r< r_0 \defn \int \log(1/\tilde \alpha_0+u)\, \d {\mb P}(u) +\log \int \tfrac{1}{(1/\tilde \alpha_0+u)} \d {\mb P}(u)$, we must have $\alpha^r\geq 1/\tilde \alpha_0>0$. Hence, for $r< r_0$ we must satisfy the KKT optimality conditions for an optimal solution $\alpha^r>0$ in Problem \eqref{eq:finite:dual}. Hence,
\[
   r = \int \log(\alpha^r+u)\, \d {\mb P}(u) +\log \int \frac{1}{\alpha^r+u} \d {\mb P}(u) = \tilde r(1/\alpha^r).
\]
Please remark that from convexity of the dual objective function in  Problem \eqref{eq:finite:dual} it also follows immediately that the function $\tilde r$ is nondecreasing. Furthermore, using the sharpened Jensen's inequality of \cite{liao2019sharpening} we get
\begin{align}
  \label{eq:tilde-r-var-lb}
  \tilde r(\tilde \alpha) =& \log \int \frac{1}{1 + \tilde \alpha u} \d \mb P(u) - \int \log\left(\frac{1}{1 + \tilde \alpha u}\right)  \d \mb P(u)
  \geq  \frac 12 \V{\mb P}{\frac{1}{1 + \tilde \alpha \zeta }}>0,
\end{align}
Consider now any nonincreasing sequence $r_k\in (0, r_0)$ and $r_k\downarrow 0$. From the fact that the function $\tilde r$ is nondecreasing we have that the auxiliary sequence $\tilde \alpha_k=1/\alpha^{r_k}$ is nonincreasing. Additionally, from the fact that $r_k=\tilde r(\tilde \alpha_k)>0$ with $\lim_{\tilde \alpha\downarrow 0}r(\tilde \alpha)=0$ (see Lemma \ref{lemma:r_lower_bound}) we must have that $\tilde \alpha_k\downarrow 0$. Moreover, observe from Proposition \ref{lemma:finite:dual} that
\[
  \frac{\V{\mb P}{\eta^{r_k}(\zeta)}}{r_k}=\frac{\V{\mb P}{\eta^{\tilde r(\tilde \alpha_k)}(\zeta)}}{\tilde r(\tilde \alpha_k)} = \frac{\V{\mb P}{\log(1/\tilde \alpha_k +\zeta)}}{\tilde r(\tilde \alpha_k)} = \frac{\V{\mb P}{\log(1 + \tilde \alpha_k\zeta )}}{\tilde r(\tilde \alpha_k)}.
\]
Thus,
\[
   \lim_{r\downarrow 0}\frac{\V{\mb P}{\eta^{r}(\zeta)}}{r} = \lim_{\tilde \alpha\downarrow 0}  \frac{\V{\mb P}{\log(1 + \tilde \alpha \zeta )}}{\tilde r(\tilde \alpha)}<\infty
\]
completing the proof.
\end{proof}

The above proof relied on two auxiliary lemmas which we state and prove next:

\begin{lemma}
  \label{lemma:r_lower_bound}
  Let $\zeta$ be regularly varying of index $\rho>1$.
  We have $\lim_{\tilde\alpha\downarrow 0}\tilde r(\tilde\alpha)=0$ not faster than
  \[
    \tilde r(\tilde \alpha)\geq  \frac{1}{32} \V{\mb P}{\min(\tilde \alpha \zeta,1)} +\tilde \alpha^{\rho}L(1/\tilde \alpha) \int^K_{1} \frac{z}{(1+z)^2 z^{\rho}} \d z + o\left(\tilde \alpha^{\rho}L(1/\tilde \alpha) \right)
  \]
  for any $K>1$. 
\end{lemma}
\begin{proof}
  We have
   \begin{align*}
    0\leq \tilde r(\tilde \alpha) \defn & \int \log(1 + \tilde \alpha u)  \d \mb P(u) + \log \int \frac{1}{1 + \tilde \alpha u} \d \mb P(u)\\
    = &  \int \tilde \alpha u \d \mb P(u) - \int (\tilde \alpha u - \log(1 + \tilde \alpha u)) \, \d \mb P(u) + \log \left( 1- \int \frac{\tilde \alpha u}{1 + \tilde \alpha u} \d \mb P(u)\right)\\
    = &  \int \tilde \alpha u \d \mb P(u) - \int \frac{\tilde \alpha u}{1 + \tilde \alpha u} \d \mb P(u) - \int (\tilde \alpha u - \log(1 + \tilde \alpha u)) \, \d \mb P(u) \\
                                    & \quad + \log \left( 1- \int \frac{\tilde \alpha u}{1 + \tilde \alpha u} \d \mb P(u)\right)+\int \frac{\tilde \alpha u}{1 + \tilde \alpha u} \d \mb P(u)\\
    = &  \int \tilde \alpha u \left(1-\frac{1}{1+\tilde \alpha u} \right) \d \mb P(u) -  \int \tilde \alpha u - \log(1 + \tilde \alpha u) \, \d \mb P(u) \\
                                        & \quad + \log \left( 1- \int \frac{\tilde \alpha u}{1 + \tilde \alpha u} \d \mb P(u)\right)+\int \frac{\tilde \alpha u}{1 + \tilde \alpha u} \d \mb P(u).
   \end{align*}
  By partial integration we have
  \begin{align*}
    \int \tilde \alpha u\left(1-\frac{1}{1+\tilde \alpha u} \right) \d \mb P(u) = & \int_0^\infty \frac{z(z+2)}{(1+z)^2} \mb P[\tilde \alpha \zeta\geq z] \d z\\
    \intertext{and likewise}
    \int (\tilde \alpha u - \log(1 + \tilde \alpha u)) \, \d \mb P(u) = & \int_0^\infty \frac{z}{1+z} \mb P[\tilde \alpha \zeta\geq z] \d z,\\
    \int \frac{\tilde \alpha u}{1 + \tilde \alpha u} \d \mb P(u)= & \int_0^\infty \frac{1}{(1+z)^2} \mb P[\tilde \alpha \zeta\geq z] \d z.
  \end{align*}
  Hence, we get  
  \begin{align*}
    \tilde r(\tilde \alpha) = & \int^\infty_{0} \frac{z}{(1+z)^2} \mb P[\tilde \alpha \zeta \geq z] \d z + \log\left(1-{\mb E}_{\mb P}\left[\frac{\tilde \alpha \zeta}{1+ \tilde \alpha \zeta}\right]\right)+{\mb E}_{\mb P} \left[\frac{\tilde \alpha \zeta}{1+\tilde\alpha \zeta}\right].
  \end{align*}
  As we have that $\log(1-x) + x \leq 0$ for all $x\in (-\infty, 1)$ it follows that
  \begin{align*}
    \tilde r(\tilde \alpha) \leq & \int^\infty_{0} \frac{z}{(1+z)^2} \mb P[\tilde \alpha \zeta\geq z] \d z 
    \leq 
    \int^\infty_{0} \frac14 \mb P[\tilde \alpha \zeta\geq z] \d z
    = 
    \frac{\tilde \alpha \E{\mb P}{\zeta}}{4}
  \end{align*}
  and hence $\lim_{\tilde \alpha\downarrow 0} \tilde r(\tilde \alpha)=0$.

  We will now establish the lower bound. First remark that 
  \begin{align*}
   \E{\mb P}{\frac{\tilde \alpha \zeta}{1+\tilde \alpha \zeta}}= & \int_0^{1} \frac{1}{(1+z)^2} \mb P[\tilde \alpha\zeta\geq z] \d z + \int_{1}^\infty \frac{1}{(1+z)^2} \mb P[\tilde \alpha \zeta\geq z] \d z.
  \end{align*}
  We will now upper bound both terms separately.
  Clearly, we have that
  \[
    \int_0^{1} \frac{1}{(1+z)^2} \mb P[\tilde \alpha\zeta\geq z] \d z \leq \int_0^1 \mb P[\tilde \alpha \zeta\geq z]\d z \leq \E{\mb P}{\tilde \alpha\zeta}.
  \]
  Furthermore, by Lemma \ref{lemma:uniform-bound} we get
  \begin{align*}
    \int_{1}^\infty \frac{1}{(1+z)^2} \mb P[\tilde \alpha \zeta\geq z] \d z  = & \int_{1}^\infty \frac{1}{(1+z)^2} \frac{L(z/\tilde \alpha)}{(z/\tilde \alpha)^\rho} \d z\\
    = & \tilde \alpha^\rho L(1/\tilde \alpha) \int_{1}^\infty \frac{1}{(1+z)^2z^\rho} \frac{L(z/\tilde \alpha)}{L(1/\tilde \alpha)} \d z\\
    \leq & \tilde \alpha^\rho L(1/\tilde \alpha) \int_{1}^\infty \frac{1}{(1+z)^2z^{(\rho-\max_{t\geq 1/\tilde \alpha}\epsilon(t))}}  \exp(\textstyle\max_{t\geq 1/\tilde \alpha} \delta(t)-\delta(1/\tilde \alpha))  \d z\\
    \leq &  \tilde \alpha^\rho L(1/\tilde \alpha) \int_{1}^\infty \frac{2}{(1+z)^2} \d z = \tilde \alpha^\rho L(1/\tilde \alpha).
  \end{align*}

  Using the Taylor expansion $\log(1-x)+x = -1/2 x^2 + o(x^2)$ we get
  \begin{align*}
    & \log\left(1-\E{\mb P}{\frac{\tilde \alpha\zeta}{1+\tilde \alpha\zeta}}\right)+\E{\mb P}{\frac{\tilde \alpha\zeta}{1+\tilde \alpha\zeta}} \\
    & \quad - \log\left(1-\int_0^{1} \frac{1}{(1+z)^2} \mb P[\tilde \alpha\zeta\geq z] \d z\right) -\int_0^{1} \frac{1}{(1+z)^2} \mb P[\tilde \alpha\zeta\geq z] \d z\\
    = & \log\frac{\left(1-\int_0^{1} \frac{1}{(1+z)^2} \mb P[\tilde \alpha\zeta\geq z] \d z - \int_1^{\infty} \frac{1}{(1+z)^2} \mb P[\tilde \alpha\zeta\geq z] \d z\right)}{\left(1-\int_0^{1} \frac{1}{(1+z)^2} \mb P[\tilde \alpha \zeta\geq z] \d z\right)}+
        \int_1^{\infty} \frac{1}{(1+z)^2} \mb P[\tilde \alpha\zeta\geq z] \d z\\
    = & \log\left(1- \frac{\int_1^{\infty} \frac{1}{(1+z)^2} \mb P[\tilde \alpha\zeta\geq z] \d z}{1-\int_0^{1} \frac{1}{(1+z)^2} \mb P[\tilde \alpha\zeta\geq z] \d z}\right)+
        \int_1^{\infty} \frac{1}{(1+z)^2} \mb P[\tilde \alpha\zeta\geq z] \d z\\
    = & \log\left(1- \frac{\int_1^{\infty} \frac{1}{(1+z)^2} \mb P[\tilde \alpha\zeta\geq z] \d z}{1-\int_0^{1} \frac{1}{(1+z)^2} \mb P[\tilde \alpha\zeta\geq z] \d z}\right)+
        \frac{\int_1^{\infty} \frac{1}{(1+z)^2} \mb P[\tilde \alpha\zeta\geq z] \d z}{1-\int_0^{1} \frac{1}{(1+z)^2} \mb P[\tilde \alpha\zeta\geq z] \d z}\\
    & \quad - \frac{\int_0^{1} \frac{1}{(1+z)^2} \mb P[\tilde \alpha\zeta\geq z] \d z \cdot\int_1^{\infty} \frac{1}{(1+z)^2} \mb P[\tilde \alpha\zeta\geq z] \d z}{1-\int_0^{1} \frac{1}{(1+z)^2} \mb P[\tilde \alpha\zeta\geq z] \d z}\\
    = & -\frac{1}{2} \left(\frac{\int_1^{\infty} \frac{1}{(1+z)^2} \mb P[\tilde \alpha\zeta\geq z] \d z}{1-\int_0^{1} \frac{1}{(1+z)^2} \mb P[\tilde \alpha\zeta\geq z] \d z} \right)^2 +o\left(\left(\frac{\int_1^{\infty} \frac{1}{(1+z)^2} \mb P[\tilde \alpha\zeta\geq z] \d z}{1-\int_0^{1} \frac{1}{(1+z)^2} \mb P[\tilde \alpha\zeta\geq z] \d z}\right)^2\right) \\
    & \quad - \int_0^{1} \frac{1}{(1+z)^2} \mb P[\tilde \alpha\zeta\geq z] \d z \cdot \left(\frac{\int_1^{\infty} \frac{1}{(1+z)^2} \mb P[\tilde \alpha\zeta\geq z] \d z}{1-\int_0^{1} \frac{1}{(1+z)^2} \mb P[\tilde \alpha\zeta\geq z] \d z}\right)\\
    = & O(\tilde \alpha^{1+\rho}L(1/\tilde \alpha)).
  \end{align*}
  Furthermore, we have for any $K\geq 1$ that 
  \begin{align*}
    & \int^\infty_{0} \frac{z}{(1+z)^2} \mb P[\tilde \alpha \zeta\geq z] \d z \\
    = & \int^{1}_{0} \frac{z}{(1+z)^2}\mb P[\tilde \alpha \zeta\geq z] \d z + \int^\infty_{1} \frac{z}{(1+z)^2} \frac{L(z/\tilde \alpha)}{(z/\tilde \alpha)^{\rho}} \d z\\
    \geq & \int^{1}_{0} \frac{z}{(1+z)^2} \mb P[\tilde \alpha \zeta\geq z] \d z + \tilde \alpha^{\rho}L(1/\tilde \alpha) \int^K_{1} \frac{z}{(1+z)^2 z^{\rho}} \frac{L(z/\tilde \alpha)}{ L(1/\tilde \alpha)} \d z.
  \end{align*}
  Recall that by the uniform convergence theorem \cite[Theorem 1.2.1]{bingham1989regular} we have for any slowly varying function $L$ that
  \[
    \lim_{\tilde \alpha\to 0} \sup_{z\in [a, b]}\abs{\frac{L(z/\tilde \alpha)}{ L(1/\tilde \alpha)}-1} = 0
  \]
  for any $0<a\leq b<\infty$.
  Hence,
\begin{align*}
  & \tilde \alpha^{\rho}L(1/\tilde \alpha) \int^K_{1} \frac{z}{(1+z)^2 z^{\rho}} \frac{L(z/\tilde \alpha)}{ L(1/\tilde \alpha)} \d z \\
  = &  \tilde \alpha^{\rho}L(1/\tilde \alpha) \int^K_{1} \frac{z}{(1+z)^2 z^{\rho}} \d z + \tilde \alpha^{\rho}L(1/\tilde \alpha) \int^K_{1} \frac{z}{(1+z)^2 z^{\rho}} \left(\frac{L(z/\tilde \alpha)}{ L(1/\tilde \alpha)}-1\right) \d z\\
  \geq & \tilde \alpha^{\rho}L(1/\tilde \alpha) \int^K_{1} \frac{z}{(1+z)^2 z^{\rho}} \d z - \tilde \alpha^{\rho}L(1/\tilde \alpha) \int^K_{1} \frac{z}{(1+z)^2 z^{\rho}} \abs{\frac{L(z/\tilde \alpha)}{ L(1/\tilde \alpha)}-1} \d z\\
  \geq & \tilde \alpha^{\rho}L(1/\tilde \alpha) \int^K_{1} \frac{z}{(1+z)^2 z^{\rho}} \d z - \tilde \alpha^{\rho}L(1/\tilde \alpha) \sup_{z\in [1, K]}\abs{\frac{L(z/\tilde \alpha)}{L(1/\tilde \alpha)}-1} \int^K_{1} \frac{z}{(1+z)^2 z^{\rho}} \d z\\
  \geq & \tilde \alpha^{\rho}L(1/\tilde \alpha) \int^K_{1} \frac{z}{(1+z)^2 z^{\rho}} \d z + o\left(\tilde \alpha^{\rho}L(1/\tilde \alpha) \right).
\end{align*}

  Chaining all results together we get
  \begin{align*}
    \tilde r(\tilde \alpha) = & \int^\infty_{0} \frac{z}{(1+z)^2} \mb P[\tilde \alpha \zeta\geq z] \d z + \log\left(1-\E{\mb P}{\frac{\tilde \alpha \zeta}{1+ \tilde \alpha \zeta}}\right)+\E{\mb P}{\frac{\tilde \alpha \zeta}{1+\tilde \alpha\zeta}}\\
    = & \int^\infty_{0} \frac{z}{(1+z)^2} \mb P[\tilde \alpha\zeta\geq z] \d z + \log\left(1-\int_0^{1} \frac{1}{(1+z)^2} \mb P[\tilde \alpha\zeta\geq z] \d z\right)+\int_0^{1} \frac{1}{(1+z)^2} \mb P[\tilde \alpha\zeta\geq z] \d z \\
                              & \quad + O\left(\tilde \alpha^{1+\rho} L(1/\tilde \alpha)\right)\\
    = &  \int^1_{0} \frac{z}{(1+z)^2} \mb P[\tilde \alpha\zeta\geq z] \d z + \log\left(1-\int_0^{1} \frac{1}{(1+z)^2} \mb P[\tilde \alpha\zeta\geq z] \d z\right)+\int_0^{1} \frac{1}{(1+z)^2} \mb P[\tilde \alpha\zeta\geq z] \d z \\
                              & \quad + \int^\infty_{1} \frac{z}{(1+z)^2} \mb P[\tilde \alpha \zeta \geq z] \d z + O\left(\tilde \alpha^{\rho+1} L(1/\tilde \alpha)\right)\\
    = &  \int \frac{z}{(1+z)^2} \mb P[\min(\tilde \alpha \zeta,1) \geq z] \d z + \log\left(1-\int \frac{1}{(1+z)^2} \mb P[\min(\tilde \alpha \zeta,1)\geq z] \d z\right) \\
                              & \quad +\int \frac{1}{(1+z)^2} \mb P[\min(\tilde \alpha \zeta,1)\geq z] \d z + \int^\infty_{1} \frac{z}{(1+z)^2} \mb P[\tilde \alpha\zeta\geq z] \d z + O\left(\tilde \alpha^{\rho+1} L(1/\tilde \alpha)\right)\\
    \geq & \V{\mb P}{\frac{1}{1+\min(\tilde \alpha \zeta,1)}}/2 + \int^\infty_{1} \frac{z}{(1+z)^2} \mb P[\tilde \alpha\zeta\geq z] \d z + O\left(\tilde \alpha^{\rho+1} L(1/\tilde \alpha)\right)\\
    \geq & \frac{1}{32} \V{\mb P}{\min(\tilde \alpha \zeta,1)} +\tilde \alpha^{\rho}L(1/\tilde \alpha) \int^K_{1} \frac{z}{(1+z)^2 z^{\rho}} \d z + o\left(\tilde \alpha^{\rho}L(1/\tilde \alpha) \right)
  \end{align*}
  where the first inequality is a sharpened Jensen's inequality of \cite{liao2019sharpening} and
  the final inequality follows from Lemma \ref{lemma:variance-formula} with $g(t)=1/(1+t)$ for $t\leq 1$ and $m=1/4$.
\end{proof}

\begin{lemma}
  \label{lemma-ubvariancelog}
    Let $\zeta$ be regularly varying of index $\rho>1$.
  We have $\lim_{\tilde \alpha\downarrow 0} \V{\mb P}{\log(1 + \tilde \alpha \zeta)}=0$ at least as fast as
  \[
    \V{\mb P}{\log(1 + \tilde \alpha \zeta)}\leq \V{\mb P}{\min (\tilde\alpha\zeta, 1))}  + \tilde \alpha^{\rho} L(1/\tilde \alpha)\int^\infty_{1} \frac{4\log(z+1)}{(z+1)z^{\rho/2}}\d z.
  \]
\end{lemma}
\begin{proof}
  By partial integration we have
\begin{align*}
  \V{\mb P}{\log(1 + \tilde \alpha \zeta)} = & \E{\mb P}{\log(1 + \tilde \alpha \zeta)^2} - \E{\mb P}{\log(1 + \tilde \alpha \zeta)}^2\\
  = & \int^\infty_{0} \frac{2\log(z+1)}{z+1} \mb P[\tilde \alpha\zeta\geq  z] \d z - \E{\mb P}{\log(1+\tilde \alpha \zeta)}^2.
\end{align*}
We can control the variance term as
\begin{align*}
  & \V{\mb P}{\log(1 + \tilde \alpha \zeta)}\\
  \leq & \int^\infty_{0} \frac{2\log(z+1)}{z+1} \mb P[\tilde \alpha \zeta\geq  z] \d z - \E{\mb P}{\log(1+\min (\tilde \alpha\zeta, 1))}^2\\
  =&  \int^1_{0} \frac{2\log(z+1)}{z+1} \mb P[\tilde \alpha \zeta\geq  z] \d z-  \E{\mb P}{\log(1+\min (\tilde \alpha\zeta, 1))}^2 + \int^\infty_{1} \frac{2\log(z+1)}{z+1} \mb P[\tilde \alpha \zeta\geq  z] \d z \\
  =&  \int \frac{2\log(z+1)}{z+1} \mb P[\min (\tilde \alpha\zeta, 1)\geq  z] \d z-  \E{\mb P}{\log(1+\min (\tilde \alpha\zeta, 1))}^2 \\
  & \quad + \int^\infty_{1} \frac{2\log(z+1)}{z+1} \mb P[\tilde \alpha \zeta\geq  z] \d z \\
  =&  \V{\mb P}{\log(1+\min (\tilde\alpha\zeta, 1))}  + \int^\infty_{1} \frac{2\log(z+1)}{z+1} \mb P[\tilde\alpha\zeta\geq  z] \d z \\
  \leq & \V{\mb P}{ \min (\tilde\alpha\zeta, 1))} + \int^\infty_{1} \frac{2\log(z+1)}{z+1} \mb P[\tilde \alpha \zeta \geq  z] \d z \\
  = & \V{\mb P}{ \min (\tilde\alpha\zeta, 1))}  + \tilde \alpha^{\rho} L(1/\tilde \alpha)\int^\infty_{1} \frac{2\log(z+1)}{(z+1)z^\rho} \frac{L(z/\tilde \alpha)}{L(1/\tilde \alpha)} \d z\\
  \leq & \V{\mb P}{\min (\tilde\alpha\zeta, 1))}  + \tilde \alpha^{\rho} L(1/\tilde \alpha)\int^\infty_{1} \frac{2\log(z+1)}{(z+1)z^{(\rho-\max_{t\geq 1/\tilde \alpha}\epsilon(t))}} \exp(\textstyle\max_{t\geq 1/\tilde \alpha} \delta(t)-\delta(1/\tilde \alpha)) \d z\\
  \leq & \V{\mb P}{\min (\tilde\alpha\zeta, 1))}  + \tilde \alpha^{\rho} L(1/\tilde \alpha)\int^\infty_{1} \frac{4\log(z+1)}{(z+1)z^{\rho/2}}\d z,
\end{align*}
where the inequality $ \V{\mb P}{\log(1+\min (\tilde\alpha\zeta, 1))}\leq \V{\mb P}{ \min (\tilde\alpha\zeta, 1))}$ follows from Lemma  \ref{lemma:variance-formula} with $g(t) = \log (1+t), t\geq 0$ and $M=1$.
\end{proof}

Unsurprisingly, we can obtain a sharper result in case $\zeta$ is bounded. Indeed, as Proposition \ref{lemma:variance-bound-bounded} points out, the ratio  $\V{\mb P}{\eta^{r}(\zeta)}/r$ is asymptotically smaller than two as $r$ gets small. Straightforward analytical calculations (which we omit for the sake of brevity) indicate that if $\zeta$ is Pareto distributed with shape parameter $\rho\in (1, 2)$ then
\[
  \lim_{r\to 0}\frac{\V{\mb P}{\eta^{r}(\zeta)}}{r} = \frac{\frac{\rho}{2-\rho}+\rho(\pi^2+12\log(2))/6}{(\rho-1) \pi/\sin(\pi(\rho-1))}>2
\]
indicating that such a result is no longer available in heavy-tailed settings. 

\begin{proposition}
  \label{lemma:variance-bound-bounded}
  Let $\zeta$ be bounded. Then,
  \[
    \lim_{r\to 0}\frac{\V{\mb P}{\eta^{r}(\zeta)}}{r} \leq 2.
  \]
\end{proposition}
\begin{proof}
  Let $0\leq \zeta\leq B<\infty$ which implies $\V{\mb P}{\zeta}\leq B^2/4<\infty$.
  Equation \eqref{eq:tilde-r-var-lb} specializes to
  \[
    \tilde r(\tilde \alpha) \geq  \frac 12 \V{\mb P}{\frac{1}{1 + \tilde \alpha \zeta }}\geq \frac {1}{2(1+\tilde \alpha B)^2} \V{\mb P}{\tilde \alpha \zeta }
  \]
  where we applied Lemma \ref{lemma:variance-formula} with the function $t\mapsto \tfrac{1}{(1+t)}$ for $0\leq t \leq \tilde \alpha B$ with $m=(1+\tilde\alpha B)^{-2}$. Furthermore, we have
  \[
    \V{\mb P}{\log(1 + \tilde \alpha \zeta)} \leq \V{\mb P}{\tilde \alpha \zeta}
  \]
  where we applied Lemma \ref{lemma:variance-formula} with the function $t\mapsto \log(1+t)$ for $0\leq t$ with $M=1$. Consequently,
  \[
    \lim_{\tilde \alpha\to 0} \frac{\V{\mb P}{\log(1 + \tilde \alpha \zeta)}}{\tilde r(\tilde \alpha)}\leq \lim_{\tilde \alpha \to 0} 2(1+\tilde \alpha B)^2 = 2.
  \]
  The statement follows by repeating verbatim the argument made in the proof of Proposition \ref{lemma:variance-bound} following Equation \eqref{eq:tilde-r-var-lb}.  
\end{proof}

\subsection{Disappointment Lower Bound}
\label{ssec:disapp-lower-bound}

We finally prove Theorem \ref{thm-klopt}, by contradiction: assume there is an estimator $\hat e$ that satisfies Equation~\eqref{eq:uniform-exponential-guarantees} and for which we have that
\begin{equation}
  \label{eq:suboptimal}
  \limsup_{n\to\infty} {\mb P}[\hat e_n > \hat e^\KL_{r'(n)}(\mb P_n)]  = p>0.
\end{equation}

That is, this estimator $\hat e$ predicts, with non-vanishing probability $p>0$, larger, less conservative outcomes than the Kullback-Leibler estimator with radius $r'(n)\geq 0$.  

Let $\delta \defn 1- \lim_{n\to\infty}\tfrac{r'(n) n}{\lambda(n)}\in (0,1)$.
Consider now $n_0\geq 1$ so that we have $r'(n)< r(n)= (1-\delta/2) \tfrac{\lambda(n)}{n} < \tfrac{\lambda(n)}{n}$ for all $n\geq n_0$.
We have following Proposition \ref{lemma:finite:dual} that
\begin{align*}
  \hat e^\KL_{r(n)}(\mb P) =  \int u \,\d \mb Q^{r(n)}(u)
\end{align*}
for a distribution $\mb Q^{r(n)}\in \mc P$ with $\KL(\mb P, \mb Q^{r(n)})  = r(n)$. The distributions $\mb Q^{r(n)}$ will serve here as frustrating distributions in the sense that we will show that Equation \eqref{eq:suboptimal} implies
\begin{equation}
  \label{eq:contradiction}
  \liminf_{n\to\infty}\frac{1}{\lambda(n)} \log \E{\mb Q^{r(n)}}{\one{\hat e_n> \E{\mb Q^{r(n)}}{\zeta} }} > -1
\end{equation}
reaching the desired contradiction.

We will do so by employing an exponential change of measure; a common tactic when establishing lower bounds in large deviation theory \cite{dembo2009large}.
Consider indeed the exponentially tilted distributions
\[
  \frac{\d \mb Q^{r(n)}_\eta}{\d \mb Q^{r(n)}}(u) = \exp\left( \eta(u) - \Lambda(\eta, \mb Q^{r(n)}) \right)\quad \forall u\geq 0
\]
with tilting parameter  $\eta:\Re_+ \to\Re \cup \{-\infty\}$ and normalization
\begin{equation}
  \label{eq:mgf}
  \Lambda(\eta, \mb Q^{r(n)})\defn\log\left(\int \exp(\eta(u)) \, \d \mb Q^{r(n)}(u)\right)
\end{equation}
which ensures that $\mb Q^{r(n)}_\eta \in \mc P$. The distribution $\mb Q^{r(n)}_\eta$ is also known as the Esscher transform \citep{escher1932probability} of $\mb Q^{r(n)}$.

To ease notation we define here $\mu^{r(n)} = \E{\mb Q^{r(n)}}{\zeta} = \hat e^\KL_{r(n)}(\mb P)$. We can derive a lower bound on the event of interest using the proposed exponential change of measure argument as follows 
\begin{align}
  & \mb Q^{r(n)}[\hat e_n > \mu^{r(n)} ]\nonumber\\
  = &  \E{{\mb Q}^{r(n)}_\eta}{\one{\hat e_n > \mu^{r(n)}}\frac{\d (\mb Q^{r(n)})^n }{\d ({\mb Q}^{r(n)}_\eta)^n} }\nonumber\\
  = &  \E{{\mb Q}^{r(n)}_\eta}{\one{\hat e_n > \mu^{r(n)}}\exp(n \Lambda(\eta, \mb Q^{r(n)}) - \textstyle n \int \eta(u) \d \mb P_n(u)) }\nonumber\\
  = & \exp(n \Lambda(\eta,  \mb Q^{r(n)}) - n \textstyle \int \eta(u) \d \mb P(u)) \cdot \E{{\mb Q}^{r(n)}_\eta}{\one{\hat e_n > \mu^{r(n)}} \exp(n  \textstyle\int \eta(u) \d[\mb P(u)-\mb P_n(u)] ) }\nonumber\\
  \geq & \exp(n [\Lambda(\eta,  \mb Q^{r(n)}) - \textstyle\int \eta(u) \d \mb P(u)] - \lambda(n)^{3/4}) \cdot \E{{\mb Q}^{r(n)}_\eta}{\one{\hat e_n > \mu^{r(n)},~ \textstyle\int  \eta(u) \d (\mb P_n-\mb P)(u) \leq \lambda(n)^{3/4}/n}}. \label{eq:tilting-lower-bound}
\end{align}

The previous lower bound holds for any tilting $\eta$ and $n\geq 1$. However, as Equation (\ref{eq:suboptimal}) only pertains to the behavior of the estimator for the distribution $\mb P$ we will try to reach our desired contradiction from Equation (\ref{eq:tilting-lower-bound}) by considering a tilting $\eta^{r(n)}$ so that we have ${\mb Q}^{r(n)}_{\eta^{r(n)}} = \mb P$. A standard result is that such tilting can be found as the log-likelihood ratio
\begin{equation}
  \label{eq:tilting}
  \eta^{r(n)}(u) = \log\left(\frac{\d \mb P}{\d \mb Q^{r(n)}}(u)\right) =\log(\alpha^{r(n)}+u) - \log(\nu^{r(n)})
\end{equation}
between $\mb P$ and $\mb Q^{r(n)}$ where the dual variables $\alpha^{r(n)}$ and $\nu^{r(n)}$ are defined in Proposition \ref{lemma:finite:dual}.
Indeed, we can verify that
\begin{align*}
  \frac{\d \mb Q^{r(n)}_{\eta^{r(n)}}}{\d \mb Q^{r(n)}}(u) = & \exp\left( \eta^{r(n)}(u) - \Lambda(\eta^{r(n)}, \mb Q^{r(n)}) \right)
  =  \frac{\exp\left( \log\left(\frac{\d \mb P}{\d \mb Q^{r(n)}}(u)\right) \right)}{\exp(\Lambda(\eta^{r(n)}, \mb Q^{r(n)}))}
  =  \frac{\frac{\d \mb P}{\d \mb Q^{r(n)}}(u)}{\int \frac{\d \mb P}{\d \mb Q^{r(n)}}(u) \, \d \mb Q^{r(n)}(u)}
  =  \frac{\d \mb P}{\d \mb Q^{r(n)}}(u).
\end{align*}
and
\begin{align*}
  \int \eta^{r(n)}(u) \d \mb P(u)- \Lambda(\eta^{r(n)}, \mb Q^{r(n)}) = & \int \log\left(\frac{\d \mb P}{\d \mb Q^{r(n)}}(u)\right) \d \mb P(u)-\log\left(\int \frac{\d \mb P}{\d \mb Q^{r(n)}}(u) \, \d \mb Q^{r(n)}(u)\right)\\
  =&\KL(\mb P, \mb Q^{r(n)})-\log\left(\int 1 \d \mb P(u)\right)=r(n).
\end{align*}

Hence, from Equation (\ref{eq:tilting-lower-bound}) with tilting $\eta^{r(n)}$ we get finally the lower bound
\begin{equation}
   \mb Q^{r(n)}[\hat e_n > \mu^{r(n)} ]
  \geq \exp(-n r(n) [1 - \lambda(n)^{3/4}/(r(n)n)]) \cdot \mb P[\hat e_n > \mu^{r(n)},~ \textstyle\int \eta^{r(n)}(u) \d (\mb P_n-\mb P)(u) \leq \lambda^{3/4}(n)/n] \label{eq:intermediate-1}
\end{equation}

We would now like to leverage Equation (\ref{eq:suboptimal}) to guarantee that the event $\hat e_n > \mu^{r(n)}$ occurs with nonzero probability.
To do so observe that we have 
\begin{align}
  {\rm{KL}}_{\inf}(\mb P_n, {\color{black}\mu^{r(n)}}) = &\left\{\begin{array}{rl}
    \min_{\mb Q} & \KL(\mb P_n, \mb Q)\\
    \st & \int u \d {\mb Q}(u)\leq \mu^{r(n)}, ~\mb Q\in \mc P\\
  \end{array}\right.\nonumber\\
  \geq &\left\{\begin{array}{rl}
    \min_{\mb Q} &  \int \eta^{r(n)}(u) \d \mb P_n(u) - \Lambda(\eta^{r(n)}, \mb Q) \\
    \st & \int u \d {\mb Q}(u)\leq \mu^{r(n)}, ~\mb Q\in \mc P
  \end{array}\right.\nonumber\\
    =& \int \eta^{r(n)}(u) \d (\mb P_n-\mb P)(u)+ \left\{\begin{array}{rl}
     \min_{\mb Q} &  \int \eta^{r(n)}(u) \d \mb P(u) - \Lambda(\eta^{r(n)}, \mb Q) \\
    \st & \int u \d {\mb Q}(u)\leq \mu^{r(n)}, ~\mb Q\in \mc P         
    \end{array}\right.\nonumber\\
  =& \int \eta^{r(n)}(u) \d (\mb P_n-\mb P)(u)+ r(n) +\left\{\begin{array}{rl} 
     \min_{\mb Q} &  - \Lambda(\eta^{r(n)}, \mb Q) \\
    \st & \int u \d {\mb Q}(u)\leq \mu^{r(n)}, ~\mb Q\in \mc P         
  \end{array}\right.\nonumber\\
  =& \int \eta^{r(n)}(u) \d (\mb P_n-\mb P)(u)+ r(n) +\left\{\begin{array}{rl} 
     \min_{\mb Q} &  - \log\left(\int \frac{\alpha^{r(n)}+u}{\nu^{r(n)}} \, \d \mb Q(u)\right) \\
    \st & \int u \d {\mb Q}(u)\leq \mu^{r(n)}=\int u \, \d \mb Q^{r(n)}(u), ~\mb Q\in \mc P         
  \end{array}\right.\nonumber\\
  \geq & \int \eta^{r(n)}(u) \d (\mb P_n-\mb P)(u)+ r(n)  - \log\left(\int \frac{\alpha^{r(n)}+u}{\nu^{r(n)}} \, \d \mb Q^{r(n)}(u)\right)\nonumber\\
  = & \int \eta^{r(n)}(u) \d (\mb P_n-\mb P)(u)+ r(n)  - \log\left(\int \frac{\d \mb P}{\d \mb Q^{r(n)}}(u) \, \d \mb Q^{r(n)}(u)\right)\nonumber\\
  = & \int \eta^{r(n)}(u) \d (\mb P_n-\mb P)(u) + r(n). \label{eq:lower-bound-KL}
\end{align}
Here, the first inequality follows from the Donsker-Varadhan representation of the Kullback-Leibler divergence \citep{donsker1975asymptotic} and where $\Lambda$ is defined in Equation (\ref{eq:mgf}). The third equality is a direct result of Equation \eqref{eq:tilting}. Finally, the second inequality follows from the fact that the logarithm is a nondecreasing function.
From inequality \eqref{eq:lower-bound-KL} we have that when the event $\int \eta^{r(n)}(u) \d (\mb P_n-\mb P)(u)\geq -\delta/4 r(n)$ occurs then 
\begin{align}
  \label{eq:KLinf_lb}
  {\rm{KL}}_{\inf}(\mb P_n, {\color{black}\mu^{r(n)}}) \geq &  \int \eta^{r(n)}(u) \d (\mb P_n-\mb P)(u) + r(n)
  \geq  -\delta/4 r(n)+r(n)= (1-\delta/4)r(n).
\end{align}
Observe furthermore that for any $\delta\in (0,1)$ we have
\[
  \lim_{n\to\infty} \frac{r'(n)}{(1-\delta/2)r(n)} = \frac{1}{(1-\delta/2)^{2}}\lim_{n\to\infty} \frac{r'(n) n}{\lambda(n)}  < \frac{1}{(1-\delta)} \lim_{n\to\infty} \frac{r'(n) n}{\lambda(n)} = 1.
\]
We may hence consider $n_0'\geq n_0$ sufficiently large so that $r'(n)\leq (1-\delta/2)r(n)$ and $\tfrac{\delta(1-\delta/2)\lambda(n)^{1/4}}{4}\geq 1$ for all $n\geq n_0'$.
Consequently, for $n\geq n'_0$ we have that the event $\int \eta^{r(n)}(u) \d (\mb P_n-\mb P)(u) \geq -\delta/4 r(n)$ implies
\begin{equation*}
  \int \eta^{r(n)}(u) \d (\mb P_n-\mb P)(u)\geq -\lambda(n)^{3/4}/n \implies \hat e^{\KL}_{r'(n)}(\mb P_n)\geq \hat e^{\KL}_{(1-\delta/2)r(n)}(\mb P_n) \geq \mu^{r(n)}.
\end{equation*}
Here the final inequality in the implication is due to inequality \eqref{eq:KLinf_lb} and $\lambda(n)^{3/4}/n\leq\delta/4 r(n)$. Combining this with our lower bound \eqref{eq:intermediate-1} gives
\begin{align*}
  \mb Q^{r(n)}[\hat e_n> \mu^{r(n)} ] \geq & \exp(-n r(n) [1 - \lambda(n)^{3/4}/(r(n)n)]) \cdot \mb P\left[\hat e_n >\hat e^{\KL}_{r'(n)}(\mb P_n) ,~ \textstyle\abs{\int \eta^{r(n)}(u) \d (\mb P_n-\mb P)(u)} \leq \lambda^{3/4}(n)/n\right]\\
  \geq & \exp(-n r(n) [1 - \lambda^{-1/4}(n)\cdot \tfrac{\lambda(n)}{(r(n)n)}]) \cdot \\
  & \quad \left(\mb P\left[\hat e_n > \hat e^{\KL}_{r'(n)}(\mb P_n)\right] + \mb P\left[\abs{ \textstyle \frac 1n\sum^n_{i=1} \eta^{r(n)}(\zeta_i) - \int \eta^{r(n)}(u) \d\mb P(u)} \leq \lambda(n)^{3/4}/n\right]-1\right).
\end{align*}

Chebyshev's inequality guarantees that
\begin{align*}
 & \mb P\left[\abs{ \textstyle \frac 1n\sum^n_{i=1} \eta^{r(n)}(\zeta_i) - \int \eta^{r(n)}(u) \d\mb P(u)} > \lambda(n)^{3/4}/n\right]\\
  \leq & \frac{\V{\mb P}{\frac 1n\sum^n_{i=1} \eta^{r(n)}(\zeta_i)}}{\lambda(n)^{3/2}/n^2}
  = \frac{\V{\mb P}{\sum^n_{i=1} \eta^{r(n)}(\zeta_i)}}{\lambda(n)^{3/2}}
  =  \frac{n\V{\mb P}{\eta^{r(n)}(\zeta)}}{\lambda(n)^{3/2}}
  =  \frac{\V{\mb P}{\eta^{r(n)}(\zeta)}}{r(n)}\frac{r(n)n}{\lambda(n)^{3/2}}
  =  \frac{\V{\mb P}{\eta^{r(n)}(\zeta)}}{r(n)}\frac{(1-\delta/2)}{\lambda(n)^{1/2}}.
\end{align*}
Hence, if $r(n)\to 0$, we have following Proposition \ref{lemma:variance-bound} that
\[
  \lim_{n\to\infty} \mb P\left[\abs{ \textstyle \frac 1n\sum^n_{i=1} \eta^{r(n)}(\zeta_i) - \int \eta^{r(n)}(u) \d\mb P(u)} > \lambda(n)^{3/4}/n\right]\leq \lim_{n\to\infty} \frac{\V{\mb P}{\eta^{r(n)}(\zeta)}}{r(n)}\frac{(1-\delta/2)}{\lambda(n)^{1/2}} = 0.
\]
If $\lim_{n\to\infty} r(n)>0$, then the previous remains true (and in fact we do not need Proposition \ref{lemma:variance-bound}). Hence, using the premise \eqref{eq:suboptimal} we now finally arrive at the claimed contradiction in Equation (\ref{eq:contradiction}). Indeed, observe
\begin{align*}
  & \liminf_{n\to\infty}\frac{1}{\lambda(n)} \log \E{\mb Q^{r(n)}}{\one{\hat e_n > \E{\mb Q^{r(n)}}{\zeta} }}\\
  \geq & \liminf_{n\to\infty} \frac{-n r(n) [1 - \lambda^{-1/4}(n)\cdot \tfrac{\lambda(n)}{(r(n)n)}]}{\lambda(n)} \\
  & \quad +  \liminf_{n\to\infty} \frac{\log \left(\mb P\left[\hat e_n > \hat e^{\KL}_{r'(n)}(\mb P_n)\right] + \mb P\left[\abs{ \textstyle \frac 1n\sum^n_{i=1} \eta^{r(n)}(\zeta_i) - \int \eta^{r(n)}(u) \d\mb P(u)} \leq \lambda(n)^{3/4}/n\right]-1\right)}{\lambda(n)}\\
  = & \liminf_{n\to\infty} \frac{-n r(n) [1 - \lambda^{-1/4}(n)\cdot \tfrac{\lambda(n)}{(r(n)n)}]}{\lambda(n)} \\
  & \quad +  \liminf_{n\to\infty} \frac{\log(p)}{\lambda(n)}\\ 
  = & \liminf_{n\to\infty} \frac{-n r(n) [1 - \frac{1}{(1-\delta/2)\lambda^{1/4}(n)}]}{\lambda(n)}\\  
  = & \liminf_{n\to\infty} \frac{-n r(n)}{\lambda(n)} \cdot \liminf_{n\to\infty} 1 - \frac{1}{(1-\delta/2)\lambda^{1/4}(n)}\\
  = & -(1-\delta/2)>-1.
\end{align*}
\qed

\bibliography{references}

\appendix

\section{Proofs}
\label{sec-proofs}

\subsection{Proofs in Section \ref{sec-challenges}}

{\color{black}
\subsubsection{Proof of Proposition \ref{prop-wass-lb}}

  The case $\gamma=0$ follows from the case $\gamma>0$ by monotonicity: $\hat e^{\W}_{n,0}\geq \hat e^{\W}_{n,\gamma'}$ for any $\gamma'>0$, so $\Pr[\hat e^{\W}_{n,0}>\mu]\geq \Pr[\hat e^{\W}_{n,\gamma'}>\mu]$.

  For $\gamma>0$, choose $\Delta>0$ such that $d(\mu+\Delta,\mu)>\gamma$, and by continuity of $d$ choose $\delta>0$ such that $d(\mu+\Delta,\mu+\delta)>\gamma$.
  Set $\mu'\defn\mu+\delta$ and define the events $E_i\defn\{\zeta_i>(\mu+\Delta)n\}$ and $E\defn\bigcup_{i=1}^n E_i$.
  For any feasible $\mb Q$ and any coupling $\mb T$ between $\mb P_n$ and $\mb Q$, let $\mb T_j$ be the conditional coupling for $\zeta_j$ with marginal mean $\mu_j\defn\int v\,\d\mb T_j(v)$.
  On $E_i$, the constraint $\frac{1}{n}\mu_i+\frac{1}{n}\sum_{j\neq i}\mu_j\leq\mu'$ with $\mu_j\geq 0$ requires $\mu_i\leq n\mu'$, and hence
  \begin{align*}
       \int d(u,v)\,\d\mb T(u,v)
    &= \frac 1n  \sum_{j=1}^n \int d(\zeta_j,v)\,\d\mb T_j(v) 
    \geq \frac{1}{n}\int d(\zeta_i,v)\,\d\mb T_i(v)\\
    &\geq \frac{1}{n}d(\zeta_i,\mu_i)
    \geq \frac{1}{n}d(\zeta_i,n\mu')
    \geq d(\zeta_i/n,\mu')
    \geq d(\mu+\Delta,\mu')
    > \gamma.
  \end{align*}
   
  The first inequality follows from $d\geq 0$; the second is Jensen's applied to the convex map $v\mapsto d(\zeta_i,v)$; the third uses $\mu_i\leq n\mu'<\zeta_i$ and that $v\mapsto d(\zeta_i,v)$ is nonincreasing on $(-\infty,\zeta_i]$ (convex with minimum $0$ at $v=\zeta_i$); the fourth uses the convexity bound $d(tu,tv)\geq t\,d(u,v)$ for $t\geq 1$ (derived below) with $t=n$; and the fifth uses $\zeta_i/n>\mu+\Delta>\mu'$ and that $u\mapsto d(u,\mu')$ is nondecreasing on $[\mu',\infty)$ (convex with minimum $0$ at $u=\mu'$).
  The bound $d(tu,tv)\geq t\,d(u,v)$ for $t\geq 1$ follows from $d(0,0)=0$: the function $\varphi(t)=d(tu,tv)$ is convex in $t$ with $\varphi(0)=0$, so $\varphi(t)/t$ is nondecreasing and $d(tu,tv)\geq t\,d(u,v)$.
  Since $d(\mu+\Delta,\mu')>\gamma$ is independent of $\mb T$ and $\mb Q$, every $\mb Q$ with $\W_d(\mb P_n,\mb Q)\leq\gamma$ has mean larger then $\mu'>\mu$, giving $\hat e^{\W}_{n,\gamma}\geq\mu'>\mu$ on $E$.
  Hence,
  \[
    \Pr\!\left[\hat e^{\W}_{n,\gamma}>\mu\right] \geq \Pr[E] = n\,{\mb P}[\zeta>(\mu+\Delta)n](1+o(1)) = \exp\!\left(-(1+o(1))(\rho-1)\log n\right),
  \]
  where the asymptotic equality uses $\Pr[E]=1-(1-{\mb P}[\zeta>(\mu+\Delta)n])^n = n\,{\mb P}[\zeta>(\mu+\Delta)n](1+o(1))$ and regular variation of $\zeta$.\qed
}

\subsubsection{Proof of Proposition \ref{prop-trim-mean}}

We only prove Part (ii), as Part (i) is a (more explicit) special case.
Using the Chernoff bound, 
\begin{equation*}
     {\mb P} [\textstyle \sum_{i=1}^n \zeta_i \wedge r(n)^{-1/a} > cnr(n)^{(a-1)/a} + \mu n ]  \leq 
     \exp\left( - n\left[s (\mu + cr(n)^{(a-1)/a}) - \log \E{}{e^{s ( \zeta \wedge r(n)^{-1/a}) }}  \right]    \right).
\end{equation*}
Now, suppose that $s = ur(n)^{1/a}$ with $0\leq u\leq a$. 
Then, using \citet[Equation (2.8)]{janssen2024fuknagaev},
\begin{align}
 {\mb E}[e^{s ( \zeta \wedge r(n)^{-1/a}) }] \leq & 1 + s\mu + s^a  \frac 2{2+a}e^a A\leq
1 + s\mu + s^a C. \label{eq-fn}
\end{align}
Plugging this into the Chernoff bound and using $\log (1+x) \leq x$, we get
\begin{align}
     {\mb P} [\textstyle\sum_{i=1}^n \zeta_i \wedge r(n)^{-1/a} > cn r(n)^{(a-1)/a} + \mu n ] \nonumber  &\leq 
     \exp\{ - n[s  cr(n)^{(a-1)/a} - s^aC]
    \}\\
    &= \exp\{ - n r(n) [uc-Cu^a]\}. \label{eq-truncub}
\end{align}
We now choose 
\begin{align*}
    c&=c_a= (a-1)^{-(a-1)/a}aC^{1/a},\\
    u& = (c/(aC))^{1/(a-1)}.
\end{align*}
Below, we show that
these choices yield  $uc-Cu^a=1$ and $u\leq a$ which was required in Equation (\ref{eq-fn}), so that the proof follows from Equation (\ref{eq-truncub}).
To show that $u\leq a$, we note that, since 
$C\geq 1/((a-1)a^a)$ by construction,
\begin{align*}
    u &=(c/(aC))^{1/(a-1)}\\ &=(a-1)^{-1/a} \left(C^{(1-a)/a}\right)^{1/(a-1)}\\ &=(a-1)^{-1/a} C^{-1/a}\\ &\leq (a-1)^{-1/a}  
    \left((a-1)a^a\right)^{1/a}\\ &=a.
\end{align*}
Finally, observe that
\[
uc = ((a-1)C)^{-1/a} (a-1)^{-(a-1)/a}a C^{1/a}=\frac a{a-1},
\]
and
\[
Cu^a = CC^{-1}(a-1)^{-1}=\frac 1{a-1},
\]
so that
$$uc-Cu^a= \frac a{a-1}- \frac 1{a-1}=1.
$$\qed

\subsubsection{Proof of Proposition \ref{prop-variance-ub}}

Remark that we have the sequence of equalities
\begin{align*}
  \Pr\left[\hat \mu_n -   \sqrt{2 r(n)} \hat\sigma_n > \mu \right] = & 
                                                                       \Pr\left[\E{\mb P_n}{\zeta} - \sqrt{2 r(n)} \sqrt{\V{\mb P_n}{\zeta}} > \mu \right]
                                                                       =  \Pr\left[\E{\mb P_n}{\tilde \zeta} - \sqrt{2 r(n)} \sqrt{\V{\mb P_n}{\tilde \zeta}} > 0 \right]\\
  = & \Pr\left[ \sqrt{n}\tfrac{\E{\mb P_n}{\tilde \zeta}}{\sqrt{\V{\mb P_n}{\tilde \zeta}}} > \sqrt{2 \lambda(n)} \right]%
\end{align*}
where the random variable $\tilde \zeta=\zeta-\mu$ has zero mean and bounded variance $\V{\mb P}{\tilde \zeta}=\V{\mb P}{ \zeta}\leq {\mb E}[\zeta^2]<\infty$.  \cite{shao1997self} defines now $S_n\defn \sum_{i=1}^n \tilde \zeta_i$ and $V_n^2 \defn\sum_{i=1}^n \tilde \zeta^2_i$ where $\tilde \zeta_i = \zeta_i-\mu$ for all $i\in [1, \dots, n]$. We have clearly
\(
n \V{\mb P_n}{\tilde \zeta} = V^2_n - S^2_n/n
\)
and consequently
\[
  \Pr\left[\hat \mu_n -   \sqrt{2 r(n)} \hat\sigma_n > \mu \right] = \Pr\left[ \tfrac{\left(\tfrac{S_n}{V_n}\right)}{\sqrt{1- \left(\tfrac{S_n}{V_n}\right)^2/n}} > \sqrt{2 \lambda(n)} \right] = \Pr\left[\tfrac{S_n}{V_n}> \sqrt{2 \lambda(n)} \sqrt{\frac{n}{n+2 \lambda(n)}} \right].
\]
Using now \citet[Theorem 3.1]{shao1997self} with $\lambda(n)\uparrow \infty$ and $\lambda(n)=o(n)$ we get
\begin{align*}
  \Pr\left[\tfrac{S_n}{V_n}> \sqrt{2 \lambda(n)} \sqrt{\frac{n}{n+2 \lambda(n)}} \right] = & \exp\left(- \frac{n\lambda(n)}{n+2 \lambda(n)}+o\left( \frac{n\lambda(n)}{n+2 \lambda(n)}\right)\right)\\
  = & \exp\left(-\lambda(n)\left( \frac{n}{n+2 \lambda(n)}+o\left( \frac{n}{n+2 \lambda(n)}\right) \right)\right)\\
  = & \exp\left(-\lambda(n)\left( 1+o(1) \right)\right)
\end{align*}
establishing the claim.\qed

\subsubsection{Proof of Proposition \ref{prop-variance-lb}}

Define $s_n(a) = a \sqrt{n/r(n)}$ and define the event 
\begin{equation}
   A_n(s,\epsilon) = \{ \zeta_n > s ; \zeta_i < \epsilon s, i<n\}. 
\end{equation}
\begin{lemma}
  \label{lemma:event-A}
    \begin{equation}
        \Pr [A_n(s_n(a), \epsilon)] = (1+o(1)) {\mb P} [ \zeta > s_n(a)].
    \end{equation}
\end{lemma}

\begin{proof}
   Write the probability as a product. 
   The term ${\mb P}[\zeta < \epsilon s_n(a)]^{n-1}$ converges to $1$, because ${\mb P}[\zeta \geq \epsilon s_n(a)] = o(1/n)$. This is in turn implied by 
   the fact that $${\mb P}[\zeta > \epsilon s_n(a)] \leq (a\epsilon)^{-2} {\mb E}[\zeta^2] r(n)/n=o(1/n)$$ as $r(n)\rightarrow 0$.
   \end{proof}

   \begin{lemma}
  \label{lemma:event-A-implies-conservatism}
\begin{equation}
 {\mb P} [\hat e^{vr}_{n}  <\mu-b \mid A_n(s_n(b (1+\epsilon)/\sqrt{2}  ), \epsilon) ] \rightarrow 1.    
\end{equation}    
\end{lemma}

\begin{proof}
    On the event $A_n(s_n(b (1+\epsilon)/\sqrt{2}), \epsilon)$, we have
     \begin{align*}
        \hat e^{vr}_{n} &=  \hat \mu_{n}  - \sqrt{2r(n)} \hat \sigma_n\\
        &\leq \hat \mu_{n}  -\sqrt{2r(n)/n} | \zeta_n - \hat \mu_n| \\
        &\leq \hat \mu_{n} (1+\sqrt{2r(n)/n} ) - \sqrt{2r(n)/n}  \zeta_n\\
        &= \hat \mu_{n-1} \frac{n-1}{n}  (1+\sqrt{2r(n)/n} ) - \zeta_n (\sqrt{2r(n)/n} - (1+\sqrt{2r(n)/n})/n)\\
        &\leq  \hat \mu_{n-1} \frac{n-1}{n}  (1+\sqrt{2r(n)/n} ) -  b (1+\epsilon) \sqrt{ n/2r(n)} (\sqrt{2r(n)/n} - (1+\sqrt{2r(n)/n})/n )\\
         &\rightarrow \mu-b(1+\epsilon)< \mu-b,
    \end{align*}
where in the second inequality we used the (triangle) inequality $-|a-b| \leq |b|-|a|$ and nonnegativity of $\zeta_n$, and the convergence follows from 
 the law of large numbers and the property $\lim_{n\to\infty}\lambda(n)=\infty$.
\end{proof}

\begin{proof}[Proof of Proposition \ref{prop-variance-lb}]
    Using symmetry and applying Lemma \ref{lemma:event-A} and \ref{lemma:event-A-implies-conservatism},  we obtain 
    \begin{align*}
     {\mb P} [\hat e^{vr}_{n} < \mu-b ] &\geq n {\mb P} [\hat e^{vr}_{n} < \mu-b ; A_n(s_n(b (1+\epsilon) / \sqrt{2}), \epsilon) ]\\
&=           n {\mb P}[ A_n(s_n(b (1+\epsilon)/\sqrt{2}), \epsilon) ]  {\mb P} [\hat e^{vr}_{n} <\mu- b \mid A_n(s_n(b (1+\epsilon)/\sqrt{2}), \epsilon)  ]\\
&= (1+o(1))  n 
{\mb P} [ \zeta > s_n( b (1+\epsilon)/\sqrt{2})]\\
& =(1+o(1))  n 
{\mb P} [ \zeta > b(1+\epsilon) \sqrt{ n/(2r(n))}]
    \end{align*}
Since the above line of argumentation is valid for every $\epsilon >0$, the proof is complete.
\end{proof}

\subsection{Proof of Proposition \ref{prop-kllb}}

We first lower bound $\hat e^{\KL}_n$ with $\hat e^{\KL, u(n)}_n$ which is obtained by truncating the revenues $\zeta_i$ as $\zeta_i \wedge u(n)$ where $u(n)\rightarrow\infty$ is  specified in more detail below.
Similarly, define the total variation estimator
\begin{align}
  \label{eq:tv-estimator}
\hat e^{\TV, u(n)}_{n} \defn & \left\{
  \begin{array}{rl}
    \min & \int u \,\d \mb Q(u)\\
    \st & \mb Q\in \mc P, ~\TV({\mb P}'_n, \mb Q)\leq \sqrt{r(n)/2}.
  \end{array}\right.
\end{align}
where $\mb P'_n$ denotes the empirical distribution of the truncated revenues $\zeta_i \wedge u(n)$ and hence also $\hat e^{\TV, u(n)}_{n}\leq \hat e^{\TV}_{n}$ where $\hat e^{\TV}_{n}$ defined as in Section \ref{ssec:wass-dro-estim}.
By Pinsker's inequality we have $\hat e^{\TV, u(n)}_{n}\leq \hat e^{\KL, u(n)}_{n}$ and hence ${\mb P}[\hat e^{\KL, u(n)}_n <a] \leq {\mb P}[\hat e^{\TV, u(n)}_{n}<a]$. A moment reflection yields that an optimal solution $\mb Q_n$ in \eqref{eq:tv-estimator} is a discrete measure on $[0, \infty)$ which removes mass $s(n) = \sqrt{r(n)/2}$ of the largest revenues scenarios and reassigns it to the zero reward scenario.
We have
$\hat e^{\TV, u(n)}_{n}\geq \frac 1n \sum_{i=1}^n  \zeta_i \wedge u(n) - u(n) s(n)$.

Combining these observations we obtain
\begin{align*}
    {\mb P}\left[  \hat e^{\KL}_n <a \right]
      &\leq {\mb P}\left[  \hat e^{\KL, u(n)}_n <a \right] \leq {\mb P}\left[  \hat e^{\TV, u(n)}_{n} <a \right] \leq  {\mb P} \left[ \frac 1n \sum_{i=1  }^n  \zeta_i \wedge u(n) < a + u(n) s(n) \right]
\end{align*}
 Now, let $u(n)\rightarrow\infty$ slow enough such that $v(n)=s(n) u(n)\rightarrow 0$. 
The proof is then  completed by applying the following lemma: 

\begin{lemma}
\label{lemma-truncatedlefttail}
    Let $u(n)\rightarrow\infty$, $v(n) \geq 0$, $v(n)\rightarrow 0$, and $b>0$. Then,
    \begin{equation*}
        {\mb P} \left[\frac 1n \sum_{i=1}^n \zeta_i \wedge u(n) < \mu-b+ v(n)  \right] = e^{-I(b)n(1+o(1)) }.
    \end{equation*}
\end{lemma} 

\begin{proof}
    First, observe that
    \begin{equation}
        {\mb P} \left[\frac 1n \sum_{i=1}^n \zeta_i \wedge u(n) < \mu-b+ v(n)  \right] \geq 
         {\mb P} \left[\frac 1n \sum_{i=1}^n \zeta_i  < \mu-b  \right] =
            e^{-I(b)n(1+o(1)) }.
    \end{equation}
    We invoke  $v(n)\geq 0, u(n) <\infty$ to obtain the inequality,  Cram\`ers theorem to get the equality. This provides an asymptotic lower bound. Recall that 
    \begin{equation*}
    I(b) := \sup_{s>0} \, (b-\mu)s-\log \E{\mb P}{ \exp(-s\zeta ) } .
\end{equation*}
{
For any $s>0$ an exponential Markov bound guarantees 
 \begin{align*}
   {\mb P} \left[\textstyle\frac 1n \sum_{i=1}^n \zeta_i \wedge u(n) < \mu-b+ v(n)  \right] = & {\mb P} \left[\exp(-s\textstyle\sum_{i=1}^n \zeta_i \wedge u(n)) > \exp(-sn(\mu-b+ v(n)) \right] \\
   \leq & \tfrac{\E{\mb P}{\exp(-s\textstyle\sum_{i=1}^n \zeta_i \wedge u(n))}}{\exp(-sn(\mu-b+ v(n))}\\
   \leq & \tfrac{\E{\mb P}{\exp(-s\textstyle (\zeta \wedge u(n)))}^n}{\exp(-sn(\mu-b+ v(n))}
    \end{align*}
    Taking logarithms and dividing by $n$ we obtain 
\begin{align*}
    \frac 1n \log  {\mb P} \left[\textstyle\frac 1n \sum_{i=1}^n \zeta_i \wedge u(n) < \mu-b+ v(n)  \right] &\leq \log \E{\mb P}{ \exp(-s (\zeta \wedge u(n)) )} -s(b-\mu-v(n)) \\
&\rightarrow \log \E{\mb P}{ \exp(-s \zeta) } -s(b-\mu).
\end{align*}
As this bound holds for any $s>0$, we have that
\[
  \limsup_{n\to\infty} \frac 1n \log  {\mb P} \left[\textstyle\frac 1n \sum_{i=1}^n \zeta_i \wedge u(n) < \mu-b+ v(n)  \right] \leq \inf_{s>0}\log \E{\mb P}{ \exp(-s \zeta) } -s(b-\mu)=-I(b)
\]
completing the proof of the upper bound.}
\end{proof}

\subsection{Supporting Results}

\begin{lemma}
  \label{lemma:variance-formula}
  Let $g$ satisfy 
  $m \abs{x-y}\leq \abs{g(x)-g(y)}\leq M \abs{x-y}$ for all $x$ and $y$. Then,
  \[
    m^2 \V{\mb P}{\zeta} \leq  \V{\mb P}{g(\zeta)} \leq M^2 \V{\mb P}{\zeta}.
  \]
\end{lemma}
\begin{proof}
  We have
  \[
    \V{\mb P}{g(\zeta)} = \frac 12 \E{\mb P\times\mb P}{(g(\zeta_1)-g(\zeta_2))^2}\leq \frac {M^2}{2} \E{\mb P\times\mb P}{(\zeta_1-\zeta_2)^2} = M^2 \V{\mb P }{\zeta}
  \]
  and
  \[
    \V{\mb P}{g(\zeta)} = \frac 12 \E{\mb P\times\mb P}{(g(\zeta_1)-g(\zeta_2))^2}\geq \frac {m^2}{2} \E{\mb P\times\mb P}{(\zeta_1-\zeta_2)^2} = m^2 \V{\mb P }{\zeta}.
  \]
\end{proof}

\begin{lemma}
  \label{lemma:uniform-bound}
  Let $L$ be slowly varying and represented as in Equation (\ref{eq:karamata-representation}).
  Then, we have
    \[
      \frac{L(z/\tilde \alpha)}{L(1/\tilde \alpha)}\leq z^{\max_{t\geq 1/\tilde \alpha}\epsilon(t)} \cdot \exp\left(\max_{t\geq 1/\tilde \alpha} \delta(t)-\delta(1/\tilde \alpha)\right)
    \]
    for all $z\geq 1$ and $\tilde \alpha \leq 1$.
  \end{lemma}
  \begin{proof}
    We have for any $z\geq 1$ that $z/\tilde \alpha\geq 1/\tilde \alpha\geq 1$.
    Hence, via the Karamata representation of $L$ we get
    \begin{align*}
      \frac{L(z/\tilde \alpha)}{L(1/\tilde \alpha)} = & \frac{\exp\left(\delta(z/\tilde \alpha)+\int_{1}^{z/\tilde \alpha} \frac{\epsilon(t)}{t} \, \d t\right)}{\exp\left(\delta(1/\tilde \alpha)+\int_{1}^{1/\tilde \alpha} \frac{\epsilon(t)}{t} \, \d t\right)} \\
      = & \exp\left( \int_1^{z/\tilde \alpha} \frac{\epsilon(t)}{t} \d t -  \int_1^{1/\tilde \alpha} \frac{\epsilon(t)}{t} \d t  + \delta(z/\tilde \alpha) - \delta(1/\tilde \alpha) \right) \\
      = & \exp\left(\int_{1/\tilde \alpha}^{z/\tilde \alpha} \frac{\epsilon(t)}{t} \d t + \delta(z/\tilde \alpha) - \delta(1/\tilde \alpha) \right)\\
      \leq & \exp\left(\max_{t\in [1/\tilde \alpha, z/\tilde \alpha]}\epsilon(t) \cdot \int_{1/\tilde \alpha}^{z/\tilde \alpha} \frac{1}{t} \d t + \delta(z/\tilde \alpha) - \delta(1/\tilde \alpha) \right)\\
      \leq & \exp\left(\max_{t\geq 1/\tilde \alpha}\epsilon(t) \cdot \int_{1/\tilde \alpha}^{z/\tilde \alpha} \frac{1}{t} \d t+ \delta(z/\tilde \alpha) - \delta(1/\tilde \alpha)\right)\\
      = & \exp\left(\max_{t\geq 1/\tilde \alpha}\epsilon(t) \cdot \log(z)+ \delta(z/\tilde \alpha) - \delta(1/\tilde \alpha)\right)\\
      \leq & z^{\max_{t\geq 1/\tilde \alpha}\epsilon(t)} \cdot \exp\left(\max_{t\geq 1/\tilde \alpha} \delta(t)-\delta(1/\tilde \alpha)\right).
    \end{align*}
    The statement is proven.
  \end{proof}  
  
\end{document}